\documentclass[12pt,reqno]{amsart}
\usepackage{amsmath,amssymb,extarrows}
\usepackage{url}
\usepackage{tikz,enumerate}
\usepackage{diagbox}
\usepackage{appendix}
\usepackage{epic}

\usepackage{float} 

\usepackage{cite}

\usepackage{hyperref}
\usepackage{array}

\usepackage{booktabs}

\allowdisplaybreaks[4]

\newtheorem{theorem}{Theorem}

\newtheorem{conj}[theorem]{Conjecture}
\newtheorem{lemma}[theorem]{Lemma}

\theoremstyle{definition}

\theoremstyle{remark}
\newtheorem{rem}{Remark}
\numberwithin{equation}{section}
\numberwithin{theorem}{section}
\numberwithin{defn}{section}

\newcommand{\tr}{\mathrm{tr}}
\newcommand{\diag}{\mathrm{diag}}

\begin{document}
\title[Proofs of Mizuno's Conjectures on Nahm Sums]{Proofs of Mizuno's Conjectures on Generalized Rank Two Nahm Sums}

\author{Boxue Wang and Liuquan Wang}
\address{School of Mathematics and Statistics, Wuhan University, Wuhan 430072, Hubei, People's Republic of China}
\email{boxwang@whu.edu.cn}
\address{School of Mathematics and Statistics, Wuhan University, Wuhan 430072, Hubei, People's Republic of China}
\email{wanglq@whu.edu.cn;mathlqwang@163.com}

\subjclass[2010]{11P84, 33D15, 33D60, 11F03}

\keywords{Rogers-Ramanujan type identities; Nahm sums; sum-product identities; vector-valued modular forms}

\begin{abstract}
 Recently, Mizuno studied generalized Nahm sums associated with symmetrizable matrices. He provided 14 sets of candidates of modular Nahm sums in rank two and justified four of them. We prove the modularity for eight other sets of candidates and present conjectural formulas for the remaining two sets of candidates. This is achieved by finding Rogers-Ramanujan type identities associated with these Nahm sums. We also prove Mizuno's conjectural modular transformation formula for a vector-valued function consists of Nahm sums. Meanwhile, we find some new non-modular identities for some other Nahm sums associated with the matrices in Mizuno's candidates.
\end{abstract}

\maketitle

\section{Introduction}\label{sec-intro}
An important problem in the theory of $q$-series and modular forms is to understand the modularity of certain $q$-hypergeometric series. In a series of papers, Nahm \cite{Nahm1994,Nahmconf,Nahm2007} considered the particular series:
\begin{align}\label{eq-Nahm}
    f_{A,b,c}(q):=\sum_{n \in \mathbb{N}^r}\frac{q^{\frac{1}{2}n^\mathrm{T}An+n^\mathrm{T}b+c}}{(q;q)_{n_1}...(q;q)_{n_r}}.
\end{align}
where $A\in \mathbb{Q}^{r\times r}$ is a positive definite matrix, $b \in \mathbb{Q}^r$ is a rational (column) vector, $c \in \mathbb{Q}$ and the shifted $q$-factorial is defined by
\begin{align}
    (a;q)_0:=1 , \quad(a;q)_n:=\prod_{k=0}^{n-1} (1-aq^k).
\end{align}
The series $f_{A,b,c}(q)$ is usually referred as a rank $r$ Nahm sum. Nahm proposed the following problem: determine all triples $(A,b,c)$ such that $f_{A,b,c}(q)$ is modular. We shall call such $(A,b,c)$ as a \textit{modular triple}. 

Nahm's motivation comes from physics and such modular Nahm sums are expected to be characters of rational conformal field theories. He made a conjecture providing a criterion for the matrix $A$ such that there exist $b$ and $c$ so that $f_{A,b,c}(q)$ is modular. 

Zagier \cite{Zagier} studied Nahm's problem systematically. He confirmed Nahm's conjecture in the rank one case by showing that there exist exactly seven modular triples:
\begin{align}
&\left(1/2,0,-1/40\right), \quad \left(1/2, 1/2, 1/40\right), \quad \left(1,0,-1/48\right), \quad \left(1,1/2,1/24\right), \nonumber \\
&\left(1,-1/2,1/24\right), \quad \left(2,0,-1/60\right), \quad \left(2,1,11/60\right).
\end{align}
In particular, the last two triples correspond to the famous Rogers-Ramanujan identities:
\begin{align}\label{RR}
    \sum_{n\ge 0}\frac{q^{n^2}}{(q;q)_n} = \frac{1}{(q,q^4;q^5)_\infty} , \quad 
    \sum_{n\ge 0}\frac{q^{n^2+n}}{(q;q)_n} = \frac{1}{(q^2,q^3;q^5)_\infty}.
\end{align}
Here and throughout this paper, we use standard $q$-series notation:
\begin{align}
 &(a;q)_\infty:=\prod_{k=0}^{\infty} (1-aq^k), \quad |q|<1,
    \\
&(a_1,\dots,a_m;q)_n:=(a_1;q)_n \cdots (a_m;q)_n, \quad n \in \mathbb{N} \cup \{\infty\} .
\end{align}

When the rank $r\geq 2$, Nahm's conjecture is known to be false in general. Vlasenko and Zwegers \cite{VZ} provided a counterexample when $r=2$. Nevertheless, Calegari, Garoufalidis and Zagier \cite{CGZ} proved that one direction of Nahm's conjecture is true. Zagier \cite{Zagier} provided a number of possible rank two and rank three modular triples. They have all been confirmed by Zagier \cite{Zagier}, Vlasenko and Zwegers \cite{VZ}, Cherednik and Feigin \cite{Feigin}, Wang \cite{Wang-rank2,Wang-rank3}, Cao, Rosengren and Wang \cite{CRW}, etc.

Recently, Mizuno \cite{Mizuno} considered generalized Nahm sums associated with symmetrizable matrices given in the following form:
\begin{align}
   \widetilde{f}_{A,b,c,d}(q):= \sum_{n \in \mathbb{N}^r} \frac{q^{\frac{1}{2}n^\mathrm{T}ADn+n^\mathrm{T}b+c}}{(q^{d_1};q^{d_1})_{n_1}...(q^{d_N};q^{d_N})_\infty}.
\end{align}
Here compared with the original Nahm sum defined in \eqref{eq-Nahm}, we have a new input $d=(d_1,\dots,d_r)\in \mathbb{Z}_{>0}^r$,  $A \in \mathbb{Q}^{r \times r} $ is a symmetrizable matrix with the symmetrizer $D := \mathrm{diag}(d_1,\dots, d_r)$ such that $AD$ is symmetric positive definite, $b \in \mathbb{Q}^r$ is a vector and  $c \in \mathbb{Q}$ is a scalar.  Such generalized Nahm sums appeared frequently in the literature and have applications in partition identities or affine Lie algebras. For example, one of Capparelli's partition identities \cite{Capparelli}, which is discovered by the theory of affine Lie algebras, states that
\begin{align}\label{eq-Capparelli}
\sum_{n_1,n_2\geq 0} \frac{q^{2n_1^2+6n_1n_2+6n_2^2}}{(q;q)_{n_1}(q^3;q^3)_{n_2}}=(-q^2,-q^3,-q^4,-q^6;q^6)_\infty.
\end{align}
This identity implies that the generalized Nahm sum $\widetilde{f}_{A,b,c,d}(q)$ associated with $A=\left(\begin{smallmatrix}
4 & 2 \\ 6 & 4
\end{smallmatrix}\right)$, $b=(0,0)^\mathrm{T}$, $c=-1/24$ and $d=(1,3)$ is modular. For convenience, when $\widetilde{f}_{A,b,c,d}(q)$ is modular we call $(A,b,c,d)$ a \textit{modular quadruple}.

Mizuno \cite{Mizuno} found that the theory of Nahm sums for symmetric matrices can be analogously applied to Nahm sums for symmetrizable matrices. For example, it seems that there are some duality principles in both cases. Zagier \cite{Zagier} conjectured the following duality between modular triples: when $(A,b,c)$ is a rank $r$ modular triple, then it is likely that
\begin{align}
(A^\star, b^\star, c^\star)=(A^{-1},A^{-1}b,\frac{1}{2}b^\mathrm{T} A^{-1}b-\frac{r}{24}-c)
\end{align}
is also a rank $r$ modular triple.  Similarly, Mizuno conjectured the following duality between modular quadruples: suppose $(A,b,c,d)$ is a modular quadruple, then $(A^*,b^*,c^*,d^*)$ is also modular quadruple where
\begin{align}
    A^*=A^{-1}, \quad b^*=A^{-1}b, \quad c^*=\frac{1}{2}b^\mathrm{T} (AD)^{-1}b-\frac{\tr D}{24}-c, \quad d^*=d.
\end{align}

Based on a numerical method from \cite{Zagier}, Mizuno \cite{Mizuno}  searched for modular quadruples $(A,b,c,d)$ when the rank $r=2,3$. He provided two lists of candidates of modular quadruples. They consist of 14 and 34 different matrices for the rank two and three cases, respectively.
This paper is devoted to  justify the modularity of Mizuno's candidates in the rank two case. The rank three case will be discussed in a forthcoming paper.

For convenience, we will call the 14 sets of candidates as Examples 1-14 according to their order of appearances in \cite[Table 1]{Mizuno}. Each example consists of a unique matrix $A$ and a vector $d$ but several choices of vectors $b$.  As a consequence of Mizuno's duality conjecture, we can classify these examples into seven groups by grouping $A$ with $A^*$ together.

Among Mizuno's 14 rank  two modular examples, four of them have been confirmed by Mizuno \cite{Mizuno} or some known identities in the literature. For instance, the example with $A=\left(\begin{smallmatrix}
4 & 2 \\ 6 & 4
\end{smallmatrix}\right)$ are justified by Capparelli's identity \eqref{eq-Capparelli}.  The examples with $A=\left(\begin{smallmatrix}
3 & 2 \\ 4 & 4
\end{smallmatrix}\right)$ and $\left(\begin{smallmatrix}
3/2 & 1/2 \\ 1 & 1
\end{smallmatrix}\right)$  were proved by Mizuno \cite{Mizuno} after finding corresponding Rogers-Ramanujan type identities. 

Mizuno also gave several comments for some of the remaining examples. For instance, though Mizuno did not justify the example with
\begin{align}\label{exam-2}
        A= \begin{pmatrix}
            1 & 1/2\\
            1 & 1
        \end{pmatrix}, \quad
        b \in \bigg\{
        \begin{pmatrix}
            0 \\ 0
        \end{pmatrix},
        \begin{pmatrix}
            0 \\ 1
        \end{pmatrix},
        \begin{pmatrix}
            1 \\ 1
        \end{pmatrix}
        \bigg\}, \quad
        d=(1,2), 
    \end{align}
he investigated the corresponding Nahm sums with certain restrictions. Specifically speaking, for $\sigma \in \{0,1\}$ we define
\begin{align}
F_\sigma(u,v;q):=\sum_{\begin{smallmatrix}
n_1\equiv \sigma \!\! \pmod{2} \\ n_1,n_2\geq 0 
\end{smallmatrix}} \frac{q^{\frac{1}{2}n_1^2+n_1n_2+n_2^2}u^{n_1}v^{n_2}}{(q;q)_{n_1}(q^2;q^2)_{n_2}}. \label{Fc-defn}
\end{align}
Let $q=e^{2\pi i\tau}$ where $\mathrm{Im} \tau>0$ and we denote $\zeta_N=e^{2\pi i/N}$ throughout this paper.  Mizuno proposed the following 
\begin{conj}\label{conj-Mizuno}
(Cf.\ \cite[Eq.\ (45)]{Mizuno}.) The following modular transformation formula holds:
\begin{align}\label{eq-conj-tran}
   U\left(-\frac{1}{\tau}\right)=\begin{pmatrix}
M & M \\ M & -M
    \end{pmatrix} U\left(\frac{\tau}{2}\right)
\end{align}
where
\begin{align}
U(\tau)=&\Big(q^{-3/56}F_0(1,1;q),q^{1/56}F_1(1,q;q),q^{9/56}F_1(q,q;q), \\
&q^{-3/56}F_1(1,1;q),q^{1/56}F_0(1,q;q),q^{9/56}F_0(q,q;q)  \Big)^\mathrm{T} \nonumber
\end{align}
and 
\begin{align}\label{eq-M-defn}
M=\begin{pmatrix}
\alpha_3 & \alpha_2 & \alpha_1 \\
\alpha_2 & -\alpha_1 & -\alpha_3 \\
\alpha_1 & -\alpha_3 & \alpha_2
\end{pmatrix}, \quad \alpha_k=\sqrt{\frac{2}{7}} \sin \frac{k\pi}{7}.
\end{align}
\end{conj}

In this paper, we will discuss Mizuno's modular examples one by one. We will prove eight of them and provide conjectural formulas for the rest. For example, we find modular product representations for each of the components in $U(\tau)$.
\begin{theorem}\label{thm-parity}
We have
\begin{align}
F_0(1,1;q)&=\frac{(-q;q^2)_\infty(q^{12},q^{16},q^{28};q^{28})_\infty}{(q^2;q^2)_\infty},   \label{r1} \\
F_1(1,1;q)&=q^{\frac{1}{2}}\frac{(-q^2;q^2)_\infty(-q,q^6,-q^7;-q^7)_\infty}{(q^2;q^2)_\infty}, \label{r2} \\
F_0(1,q;q)&=\frac{(-q^2;q^2)_\infty(-q^3,q^4,-q^7;-q^7)_\infty}{(q^2;q^2)_\infty}, \label{r3} \\
F_1(1,q;q)&=q^{\frac{1}{2}}\frac{(-q;q^2)_\infty(q^8,q^{20},q^{28};q^{28})_\infty}{(q^2;q^2)_\infty}, \label{r4} \\
F_0(q,q;q)&=
\frac{(-q^2;q^2)_\infty(q^2,-q^5,-q^7;-q^7)_\infty}{(q^2;q^2)_\infty}, \label{r5} \\
F_1(q,q;q)&=
q^{\frac{3}{2}}\frac{(-q;q^2)_\infty(q^4,q^{24},q^{28};q^{28})_\infty}{(q^2;q^2)_\infty}.\label{r6}
\end{align}
\end{theorem}
Based on this theorem, we confirm the conjectural formula \eqref{eq-conj-tran}.
\begin{theorem}\label{thm-conj}
Conjecture \ref{conj-Mizuno} holds. Moreover, we have
\begin{align}
U(\tau+1)=\Lambda^4 U(\tau)  , \quad U\left(\frac{\tau}{2\tau+1}\right)=  \begin{pmatrix}
0 & P^3WP \\ P^3W^\mathrm{T}P & 0
\end{pmatrix}U(\tau)
\end{align}
where 
\begin{align}
& P=\diag\left(\zeta_{112}^{-3},\zeta_{112},\zeta_{112}^9 \right),  \quad \Lambda=\begin{pmatrix}
    P & 0 \\ 0 & P
\end{pmatrix},\\
&W=\sqrt{2}\begin{pmatrix}
        \alpha_1 & \alpha_3 & \alpha_2 \\
        -\alpha_3 & \alpha_2 & -\alpha_1 \\
     \alpha_2 &  \alpha_1 & -\alpha_3
    \end{pmatrix}. 
\end{align}
As a consequence, $U(\tau)$ is a vector-valued modular function on $\Gamma_0(2)$.
\end{theorem}

To summarize, among the 14 rank two modular examples found by Mizuno, only two of them are left open. They are associated with the matrices $A=\left(\begin{smallmatrix}
2 & 1  \\ 3 & 2
\end{smallmatrix}\right)$ and its inverse $A^{-1}=\left(\begin{smallmatrix} 2 & -1 \\ -3 & 2 \end{smallmatrix}\right)$. As pointed out by Mizuno, the first one corresponds to the open conjectures of Kanade and Russell \cite{K2015,Kursungoz-AC}. Though we were not able to prove them, we provide equivalent conjectural product representations for these Nahm sums. We also find some conjectural identities for the example associated with $A^{-1}$, which justify its modularity.

Furthermore, we searched for identities of the form
\begin{align}\label{id-form}
\widetilde{f}_{A,b,0,d}(q)=q^\delta \prod\limits_{n=1}^\infty (1-q^n)^{a_n},
\end{align}
where $\delta$ is a rational scalar, $\{a_n\}$ is some bounded sequence of integers with $(A,d)$ from Mizuno's rank two examples and we allow $b=(b_1,b_2)^\mathrm{T}$ to change. With the help of Maple we discover some new identities. Surprisingly, the duality predicated by Mizuno still exists. That is, once we find a product representation for $\widetilde{f}_{A,b,0,d}$ in the form \eqref{id-form}, then we have a similar dual identity for $\widetilde{f}_{A^*,b^*,0,d^*}$. 
For example, for the following choices of $(A,b,d)$ where $A$ and $d$ are the same with Example 1:
\begin{align}\label{new-exam-1}
             A = \begin{pmatrix}
        2 & 1\\
        2 & 2
\end{pmatrix}, \quad b=\begin{pmatrix}
    (a-3)/2 \\ a
\end{pmatrix}, \quad d=(1,2),
\end{align}
we find that (see Theorem \ref{thm-new-exam1})
\begin{align}
\sum_{i,j\ge 0}\frac{q^{2i^2+4ij+4j^2+(a-3)i+2aj}}{(q^2;q^2)_i(q^4;q^4)_j}
=(1+q^{a+1}+q^{a-1})(-q^{a+3};q^2)_\infty.
\end{align}
As its dual, for
\begin{align}\label{new-exam-1-dual}
    A=\begin{pmatrix}
            1 & -1/2\\
            -1 & 1
        \end{pmatrix}, \quad b=\begin{pmatrix} -3/2 \\ (a+3)/2 \end{pmatrix}, \quad d=(1,2), \quad a\in \mathbb{Q}
\end{align}
we prove the following identity (see Theorem \ref{thm-new-exam2}):
\begin{align}
\sum_{i,j\geq 0} \frac{q^{i^2-2ij+2j^2-3i+(a+3)j}}{(q^2;q^2)_i(q^4;q^4)_j}=2q^{-2}(1+q^{a+1}+q^2)(-q^2,-q^{a+3};q^2)_\infty.
\end{align}

The rest of this paper is organized as follows. In Section \ref{sec-pre} we collect some auxiliary results needed in the proofs. In Section \ref{sec-exam} we discuss Mizuno's rank two modular examples one by one and provide proofs or conjectural formulas for them. In the same time, we will present some new identities in the form \eqref{id-form} which are companions to Mizuno's examples.

\section{ Preliminaries}\label{sec-pre}
We need Euler’s $q$-exponential identities \cite[Corollary 2.2]{Andrews}
\begin{align}\label{euler}
\sum_{n=0}^{\infty}\frac{z^n}{(q;q)_n}
=
\frac{1}{(z;q)_{\infty}}, \quad|z|<1, \quad
\sum_{n=0}^{\infty}\frac{q^{(^n_2)} z^n}{(q;q)_n}
=
(-z;q)_{\infty}
\end{align}
and the Jacobi triple product identity \cite[Theorem 2.8]{Andrews}
\begin{align}\label{Jacobi}
(q,z,q/z;q)_\infty=\sum_{n=-\infty}^\infty (-1)^nq^{\binom{n}{2}}z^n.
\end{align}

For Example 4 we need to evaluate contour integrals of some infinite products. For this we rely on the following result found from the book of  Gasper and Rahman \cite{GR-book}. Note that the symbol ``idem $(c_1;c_2,\dots,c_C)$'' after an expression stands for the sum of the $(C-1)$ expressions obtained from the preceding expression by interchanging $c_1$ with each $c_k$, $k=2,3,\dots,C$.

\begin{lemma}\label{lem-integral}
(Cf.\ \cite[Eq.\ (4.10.6)]{GR-book})
Suppose that
$$P(z):=\frac{(a_1z,\dots,a_Az,b_1/z,\dots,b_B/z;q)_\infty}{(c_1z,\dots,c_Cz,d_1/z,\dots,d_D/z;q)_\infty}$$
has only simple poles. We have
\begin{align}\label{eq-integral}
\oint P(z)\frac{dz}{2\pi iz}=& \frac{(b_1c_1,\dots,b_Bc_1,a_1/c_1,\dots,a_A/c_1;q)_\infty }{(q,d_1c_1,\dots,d_Dc_1,c_2/c_1,\dots,c_C/c_1;q)_\infty} \nonumber \\
& \times \sum_{n=0}^\infty \frac{(d_1c_1,\dots,d_Dc_1,qc_1/a_1,\dots,qc_1/a_A;q)_n}{(q,b_1c_1,\dots,b_Bc_1,qc_1/c_2,\dots,qc_1/c_C;q)_n} \nonumber \\
&\times \Big(-c_1q^{(n+1)/2}\Big)^{n(C-A)}\Big(\frac{a_1\cdots a_A}{c_1\cdots c_C} \Big)^n +\text{idem} ~(c_1;c_2,\dots,c_C)
\end{align}
when $C>A$, or if $C=A$ and
\begin{align}\label{cond}
\left|\frac{a_1\cdots a_A }{c_1\cdots c_C}\right|<1.
\end{align}
Here the integration is over a positively oriented contour so that the poles of
$$(c_1z,\dots,c_Cz;q)_\infty^{-1}$$
lie outside the contour, and the origin and poles of $(d_1/z,\dots,d_D/z;q)_\infty^{-1}$ lie inside the contour.
\end{lemma}


Let $J_{a,m}:=(q^a,q^{m-a},q^m;q^m)_\infty$ and $J_m:=(q^m;q^m)_\infty$. We will reduce some double sums to single sums, and most of these single sums appeared in Slater's list \cite{Slater}  of Rogers-Ramanujan type identities. For convenience, we label the $n$-th identity in Slater's list as (S. $n$). Below we list the identities we need:
\begin{align}
&\sum_{n=0}^\infty \frac{q^{n(n+1)/2}(a;q)_n}{(q;q)_n}=(aq;q^2)_\infty (-q;q)_\infty,  \quad \text{(\cite[p.\ 21, Cor.\ 2.7]{Andrews})} \label{Lebesgue} \\
&\sum_{n=0}^\infty \frac{q^{n^2+n}(-q;q)_{n+1}}{(q^2;q^2)_n}=\frac{J_5}{J_{1,5}}, \quad \text{(\cite[Eq.\ (2.5.4)]{MSZ})} \label{RR-variant} \\
&\sum_{n=0}^\infty\frac{q^{n(n+1)}(-q;q^2)_n}{(q^2;q^2)_n}
=
\frac{1}{(q^2,q^3,q^7;q^8)_\infty}, \quad \text{(\cite[Eq.\ (2.24)]{Go})}  \label{Go2.24}\\
&\sum_{n\ge 0}\frac{q^{n(n+2)}}{(q^4;q^4)_n}
    =
    \frac{J_{1,5}}{J_{1,4}} ,  \quad   \text{(S.\ 16)}
    \label{S16}
    \\
&\sum_{n\ge 0}\frac{q^{n^2}}{(q^4;q^4)_n}
    =
     \frac{J_{2,5}}{J_{1,4}} ,   \quad  \text{(S.\ 20)}
     \label{S20} \\
&\sum_{n=0}^\infty \frac{q^{2n^2+2n}}{(-q;q)_{2n+1}(q^2;q^2)_n}=\frac{J_{1,7}}{J_2}, \quad \text{(S. 31)}  \label{S31} \\
&\sum_{n=0}^\infty \frac{q^{2n^2+2n}}{(-q;q)_{2n}(q^2;q^2)_n}=\frac{J_{2,7}}{J_2}, \quad \text{(Rogers \cite[p.\ 342]{Rogers1894}, S. 32)} \label{S32} \\
&\sum_{n=0}^\infty \frac{q^{2n^2}}{(-q;q)_{2n}(q^2;q^2)_n}=\frac{J_{3,7}}{J_2}, \quad \text{(S. 33)} \label{S33} \\
&\sum_{n=0}^\infty\frac{(-q;q^2)_nq^{n^2}}{(q^2;q^2)_n}=\frac{1}{(q,q^4,q^7;q^8)_\infty}, \quad \text{(S. 36)} \label{S36}\\
&\sum_{n=0}^\infty \frac{q^{n^2}(-q^3;q^6)_n}{(q^2;q^2)_{2n}}=\frac{J_2J_{2,24}J_{10,24}}{J_1J_{24}J_{4,24}}, \quad \text{(Ramanujan \cite[Entry 5.3.8]{RamaLost2})} \label{Rama538}  \\
&\sum_{n=0}^\infty \frac{q^{n(n+1)/2}(-q^3;q^3)_n}{(q;q)_{2n+1}}=\frac{J_{12}^5J_{2,12}}{J_{1,12}^2J_{3,12}^2J_{5,12}^2}, \quad \text{(\cite[Eq.\ (1.24)]{MS})} \label{MSZ124} \\
&\sum_{n=0}^\infty \frac{q^{n^2+2n}}{(q;q)_n(q;q^2)_{n+1}}=\frac{J_{2,14}}{J_1}, \quad \text{(Rogers \cite[p.\ 329 (1)]{Rogers1917}, S. 59)} \label{S59} \\
&\sum_{n=0}^\infty \frac{q^{n^2+n}}{(q;q)_n(q;q^2)_{n+1}}=\frac{J_{4,14}}{J_1}, \quad \text{(S. 60)} \label{S60} \\
&\sum_{n=0}^\infty \frac{q^{n^2}}{(q;q^2)_n(q;q)_n}=\frac{J_{6,14}}{J_1},  \quad \text{(S. 61)}\label{S61} \\
&\sum_{n=0}^\infty \frac{q^{(n^2+n)/2}}{(q;q)_n(q;q^2)_{n+1}}=\frac{J_2J_{14}^3}{J_1J_{1,14}J_{4,14}J_{6,14}}, \quad \text{(S. 80)} \label{S80} \\
&\sum_{n=0}^\infty \frac{q^{(n^2+n)/2}}{(q;q^2)_n(q;q)_n}=\frac{J_2J_{14}^3}{J_1J_{2,14}J_{3,14}J_{4,14}}, \quad \text{(S. 81)}  \label{S81}\\
&\sum_{n=0}^\infty \frac{q^{(n^2+3n)/2}}{(q;q)_n(q;q^2)_{n+1}}=\frac{J_2J_{14}^3}{J_1J_{2,14}J_{5,14}J_{6,14}}, \quad \text{(S. 82)} \label{S82} \\
&\sum_{n=0}^\infty \frac{q^{n^2}}{(q;q^2)_n(q^4;q^4)_n}=\frac{J_2J_{14}J_{3,28}J_{11,28}}{J_1J_{28}J_{4,28}J_{12,28}}, \quad \text{(S. 117)}  \label{S117} \\
&\sum_{n=0}^\infty \frac{q^{n^2+2n}}{(q;q^2)_n(q^4;q^4)_n}=\frac{J_2J_{1,14}J_{12,28}}{J_1J_4J_{28}}, \quad \text{(S. 118)} \label{S118} \\ 
&\sum_{n=0}^\infty \frac{q^{n^2+2n}}{(q;q)_{2n+1}(-q^2;q^2)_n}=\frac{J_2J_{4,28}J_{5,14}}{J_1J_4J_{28}}. \quad \text{(S. 119)} \label{S119} 
\end{align}
In our second proof of Example 6, we need to use the following two identities.
\begin{lemma}\label{lem-RR-id}
We have
\begin{align}
&\sum_{n=0}^\infty \frac{q^{3n^2}(-q;q^2)_{3n}}{(q^6;q^6)_{2n}}=\frac{J_{24}^3}{J_{3,24}J_{4,24}J_{9,24}}, \label{lem-W-1} \\
&\sum_{n=0}^\infty \frac{q^{3n(n+1)/2}(-q;q)_{3n+1}}{(q^3;q^3)_{2n+1}}=\frac{J_{12}^3J_{2,12}}{J_{1,12}J_{3,12}^2J_{5,12}}. \label{lem-W-2}
\end{align}
\end{lemma}
The above two identities appear to be new as we cannot find a reference. Our proof requires the use of the ${}_r\phi_s$ series defined by 
\begin{align}
{}_r\phi_s \bigg(\genfrac{}{}{0pt}{}{a_1,  \cdots,  a_r}{b_1,   \dots,  b_s}; q,z \bigg) 
:=\sum_{n=0}^\infty\frac{(a_1,\cdots,a_r;q)_n}{(q,b_1,\cdots,b_s;q)_n}((-1)^nq^{n(n-1)/2})^{1+s-r}z^n.
\end{align}
Moreover, we need the $q$-Kummer (Bailey-Daum) summation formula \cite[(\uppercase\expandafter{\romannumeral2}. 9)]{GR-book}
\begin{align}\label{eq-BD}
    {}_2\phi_1 \bigg(\genfrac{}{}{0pt}{}{a, b}{aq/b}; q,-q/b\bigg)=\frac{(-q;q)_\infty (aq,aq^2/b^2;q^2)_\infty}{(-q/b,aq/b;q)_\infty}
\end{align}
and the following transformation formula \cite[(\uppercase\expandafter{\romannumeral3}.9)]{GR-book}
\begin{align}\label{eq-3phi2}
    {}_3\phi_2 \bigg(\genfrac{}{}{0pt}{}{a, b,c}{d,   e}; q, \frac{de}{abc}\bigg)=\frac{(e/a,de/bc;q)_\infty}{(e,de/abc;q)_\infty}  {}_3\phi_2 \bigg(\genfrac{}{}{0pt}{}{a,d/b,d/c}{d, de/bc}; q, \frac{e}{a}\bigg).
\end{align}
\begin{proof}[Proof of Lemma \ref{lem-RR-id}]
We have
\begin{align*}
&\sum_{n=0}^\infty \frac{q^{3n^2}(-q;q^2)_{3n}}{(q^6;q^6)_{2n}}=\lim\limits_{a\rightarrow 0} \sum_{n=0}^\infty \frac{(-q,-q^3,-q^5,-q^3/a;q^6)_na^n}{(q^3,-q^3,q^6,-q^6;q^6)_n}\\ 
&=\lim\limits_{a\rightarrow 0} {}_3\phi_2 \bigg(\genfrac{}{}{0pt}{}{-q, -q^5,-q^3/a}{q^3,-q^6}; q^6,a\bigg) \\
&=\lim\limits_{a\rightarrow 0} \frac{(q^5,-aq;q^6)_\infty}{(-q^6,a;q^6)_\infty} {}_3\phi_2 \bigg(\genfrac{}{}{0pt}{}{-q,-q^{-2},-a}{q^3,-aq}; q^6,q^5\bigg) \quad \text{(by \eqref{eq-3phi2})} \\
&=\frac{(q^5;q^6)_\infty}{(-q^6;q^6)_\infty} {}_2\phi_1 \bigg(\genfrac{}{}{0pt}{}{-q^{-2}, -q}{q^3}; q^6,q^5\bigg) \\
&=\frac{(q^5;q^6)_\infty}{(-q^6;q^6)_\infty} \cdot \frac{(-q^6;q^6)_\infty (-q^4,-q^8;q^{12})_\infty}{(q^5,q^3;q^6)_\infty} \quad \text{(by \eqref{eq-BD})} \\
&=\frac{(-q^4,-q^8;q^{12})_\infty}{(q^3;q^6)_\infty}.
\end{align*}
This proves \eqref{lem-W-1}.

Next, we have 
\begin{align*}
&\sum_{n=0}^\infty \frac{q^{3n(n+1)/2}(-q;q)_{3n+1}}{(q^3;q^3)_{2n+1}}=\frac{1+q}{1-q^3}\lim\limits_{a\rightarrow 0}\sum_{n=0}^\infty \frac{(-q^2,-q^3,-q^4,-q^3/a;q^3)_na^n}{(q^3,-q^3,q^{9/2},-q^{9/2};q^3)_n} \nonumber \\
&=\frac{1+q}{1-q^3}\lim\limits_{a\rightarrow 0} {}_3\phi_2 \bigg(\genfrac{}{}{0pt}{}{-q^2,-q^4,-q^3/a}{q^{9/2},-q^{9/2}}; q^3,a\bigg) \nonumber \\
&=\frac{1+q}{1-q^3}\cdot \lim\limits_{a\rightarrow 0} \frac{(q^{5/2},-aq^2;q^3)_\infty}{(-q^{9/2},a;q^3)_\infty} {}_3\phi_2 \bigg(\genfrac{}{}{0pt}{}{-q^2,-q^{1/2},-q^{3/2}a}{q^{9/2},-aq^2}; q^3,q^{5/2}\bigg) \quad \text{(by \eqref{eq-3phi2})} \\
&=\frac{1+q}{1-q^3}\cdot \frac{(q^{5/2};q^3)_\infty}{(-q^{9/2};q^3)_\infty} {}_2\phi_1 \bigg(\genfrac{}{}{0pt}{}{-q^2,-q^{1/2}}{q^{9/2}}; q^3,q^{5/2}\bigg) \nonumber \\
&=\frac{1+q}{1-q^3}\cdot \frac{(q^{5/2};q^3)_\infty}{(-q^{9/2};q^3)_\infty}\cdot \frac{(-q^3;q^3)_\infty (-q^5,-q^7;q^6)_\infty}{(q^{5/2},q^{9/2};q^3)_\infty} \quad \text{(by \eqref{eq-BD})}\\
&=\frac{(-q^3;q^3)_\infty (-q,-q^5;q^6)_\infty}{(q^3;q^6)_\infty}.
\end{align*}
This proves \eqref{lem-W-2}.
\end{proof}

In order to prove the modular transformation formula \eqref{eq-conj-tran}, we need some knowledge from the theory of modular forms. Recall the full modular group 
$$\mathrm{SL}(2,\mathbb{Z})=\left\{\begin{pmatrix} a & b \\ c & d \end{pmatrix}: a,b,c,d\in \mathbb{Z}, ad-bc=1\right\}$$
and its congruence subgroup
\begin{align*}
&\Gamma_0(N):=\left\{\begin{pmatrix} a & b \\ c & d  \end{pmatrix} \in \mathrm{SL}(2,\mathbb{Z}):  \begin{pmatrix} a & b \\ c & d \end{pmatrix} \equiv \begin{pmatrix} * & * \\ 0 & * \end{pmatrix} \pmod{N}\right\}.
\end{align*}

We define the Dedekind eta function
\begin{align}
\eta(\tau):=q^{1/24}(q;q)_\infty
\end{align}
and Weber's modular functions \cite{We}:
\begin{align}\label{Weber-defn}
\mathfrak{f}(\tau):=q^{-1/48} (-q^{1/2};q)_\infty, \ \  \mathfrak{f}_1(\tau):=q^{-1/48} (q^{1/2};q)_\infty, \ \ \mathfrak{f}_2(\tau):=q^{1/24}(-q;q)_\infty.
\end{align}
The following properties are well-known:
\begin{align}
&\eta(-1/\tau)=\sqrt{-i \tau} \eta(\tau), \ \ \eta(\tau+1)=e^{\pi i  /12} \eta(\tau), \label{eta-tran} \\
&\mathfrak{f}(-1/\tau) = \mathfrak{f}(\tau), \ \ \ \mathfrak{f}_2(-1/\tau) = \frac{1}{\sqrt{2}}  \mathfrak{f}_1(\tau), \ \ \mathfrak{f}_1(-1/\tau)=\sqrt{2} \mathfrak{f}_2(\tau), \label{Weber-1} \\
&\mathfrak{f}(\tau+1) =e^{-\pi i /24}  \mathfrak{f}_1(\tau), \ \ \mathfrak{f}_1(\tau+1) = e^{-\pi i/24} \mathfrak{f}(\tau),  \ \ \mathfrak{f}_2(\tau+1) = e^{\pi i/12} \mathfrak{f}_2(\tau). \label{Weber-2} 
\end{align}

For $m\in \frac{1}{2}\mathbb{N}$ and $j\in \frac{1}{2}\mathbb{Z}$, we define the theta series \cite[p.\ 215]{Wakimoto}
\begin{align}
    h_{j,m}(\tau):=\sum_{k\in \mathbb{Z}} q^{m(k+\frac{j}{2m})^2}, \quad
    g_{j,m}(\tau):=\sum_{k\in \mathbb{Z}} (-1)^kq^{m(k+\frac{j}{2m})^2}.
\end{align}
It is easy to verify the following properties \cite[p.\ 215]{Wakimoto}:
\begin{align}
&h_{j,m}(\tau)=h_{-j,m}(\tau)=h_{2m+j,m}(\tau), \quad g_{j,m}(\tau)=g_{-j,m}(\tau)=-g_{2m+j,m}(\tau), \label{g-h-period}\\
&g_{j,m}(\tau)=h_{2j,4m}(\tau)-h_{4m-2j,4m}(\tau), \label{g-h-change}\\
&h_{j,m}(2\tau)=h_{2j,2m}(\tau), \quad g_{j,m}(2\tau)=g_{2j,2m}(\tau). \label{hg-double}
\end{align}

\begin{lemma}\label{lem-modular}
(Cf. \cite[p.\ 215, Theorem 4.5]{Wakimoto}.) For $j\in \mathbb{Z}$ and $m\in \frac{1}{2}\mathbb{N}$ we have
\begin{align}
h_{j,m}\left(-\frac{1}{\tau}\right)&=\frac{(-i\tau)^{\frac{1}{2}}}{\sqrt{2m}} \sum_{0\leq k\leq 2m-1} e^{\frac{\pi ijk}{m}}h_{k,m}(\tau), \\
g_{j,m}\left(-\frac{1}{\tau}\right)&=\frac{(-i\tau)^{\frac{1}{2}}}{\sqrt{2m}}\sum_{\begin{smallmatrix}
    0\leq k\leq 4m-1 \\ k ~~\text{odd}
\end{smallmatrix}} e^{\frac{\pi ijk}{2m}} h_{\frac{k}{2},m}(\tau).
\end{align}
For $j+m\in \mathbb{Z}$ we have
\begin{align}\label{fg-translation}
h_{j,m}(\tau+1)=e^{\frac{\pi ij^2}{2m}}h_{j,m}(\tau), \quad g_{j,m}(\tau+1)=e^{\frac{\pi ij^2}{2m}}g_{j,m}(\tau).
\end{align}
\end{lemma}

\section{Identities for Mizuno's rank two examples}\label{sec-exam}
Since the two examples in the $i$-th row ($1\leq i\leq 7$) in \cite[Table 1]{Mizuno} are dual to each other, they will be studied together. Since $c$ is uniquely determined by $A,b,d$ when $(A,b,c,d)$ is a modular quadruple, we will not mention it.

\subsection{Examples 1 and 2}
Example 1 corresponds to
\begin{align}\label{exam-1}
    A = \begin{pmatrix}
        2 & 1\\
        2 & 2
    \end{pmatrix}, \quad 
    b \in \bigg\{
    \begin{pmatrix}
        0\\0
    \end{pmatrix},
    \begin{pmatrix}
        0\\1
    \end{pmatrix},
    \begin{pmatrix}
        -1\\-1
    \end{pmatrix}, 
    \begin{pmatrix}
        -1/2\\0
    \end{pmatrix},
    \begin{pmatrix}
        1\\2
    \end{pmatrix}\bigg\}, \quad d=(1,2).
\end{align}
The modularity of this example follows from the following known identities:
    \begin{align}
        \sum_{i,j\ge 0}\frac{q^{i^2+2ij+2j^2}}{(q;q)_i(q^2;q^2)_j}
        &=\frac{(q^3;q^3)_\infty^2}{(q;q)_\infty(q^6;q^6)_\infty} ,\label{t1-1-1}
        \\
        \sum_{i,j\ge 0}\frac{q^{i^2+2ij+2j^2+j}}{(q;q)_i(q^2;q^2)_j}
        &=\frac{1}{(q;q^2)_\infty} , \label{t1-1-2}
        \\
        \sum_{i,j\ge 0}\frac{q^{i^2+2ij+2j^2-i-j}}{(q;q)_i(q^2;q^2)_j}
        &=(-1;q)_\infty , \label{t1-1-3}
        \\
        \sum_{i,j\ge 0}\frac{q^{2i^2+4ij+4j^2-i}}{(q^2;q^2)_i(q^4;q^4)_j}
        &=(-q;q^2)_\infty , \label{t1-1-4}
        \\
        \sum_{i,j\ge 0}\frac{q^{i^2+2ij+2j^2+i+2j}}{(q;q)_(q^2;q^2)_j}
        &=\frac{(q^6;q^6)^2}{(q^2;q^2)_\infty(q^3;q^3)_\infty} . \label{t1-1-5}
    \end{align}
The identities  \eqref{t1-1-1}, \eqref{t1-1-2} and \eqref{t1-1-5} are special instances of Bressoud's identities \cite{Bressoud1979}. The identities \eqref{t1-1-3} and  \eqref{t1-1-4} follow from the following identity of Cao and Wang \cite[Eq.\ (3.29)]{CW}:
\begin{align}\label{eq-Cao-Wang}
\sum_{i,j\geq 0} \frac{u^{i+2j}q^{i^2+2ij+2j^2-i-j}}{(q;q)_i(q^2;q^2)_j}=(-u;q)_\infty.
\end{align}
Note that with $u=q^a$ ($a\in \mathbb{Q}$) the left side is the Nahm sum corresponding to 
\begin{align}\label{data-exam1-Cao-Wang}
    A = \begin{pmatrix}
        2 & 1\\
        2 & 2
    \end{pmatrix}, \quad 
    b= \begin{pmatrix}
       a-1 \\ 2a-1
    \end{pmatrix}, \quad d=(1,2).
\end{align}

Example 2 is dual to Example 1, and it corresponds to
    \begin{align}\label{dual-exam-1}
       & A=\begin{pmatrix}
            1 & -1/2\\
            -1 & 1
        \end{pmatrix}, ~~
        b \in \bigg\{ \begin{pmatrix}
            0 \\ 0
        \end{pmatrix},
        \begin{pmatrix}
            -1/2 \\ 1
        \end{pmatrix}, 
        \begin{pmatrix}
            -1/2 \\ 0
        \end{pmatrix},
        \begin{pmatrix}
            -1/2 \\ 1/2
        \end{pmatrix} ,
        \begin{pmatrix}
            0 \\ 1
        \end{pmatrix}
        \bigg\}, ~~
    d=(1,2). 
    \end{align}
The modularity follows from the following identities:
        \begin{align}
            \sum_{i,j\ge 0} \frac{q^{i^2-2ij+2j^2}}{(q^2;q^2)_i(q^4;q^4)_j}  &=\frac{(q^2;q^2)_\infty^3(q^3;q^3)_\infty^2}{(q;q)_\infty^2(q^4;q^4)_\infty^2(q^6;q^6)_\infty} , \label{t1-2-1}
            \\
            \sum_{i,j\ge 0}\frac{q^{\frac{1}{2}i^2-ij+j^2-\frac{1}{2}i+j}}{(q;q)_i(q^2;q^2)_j}
            &=(-1,-q;q)_\infty , \label{t1-2-2}
            \\
            \sum_{i,j\ge 0}\frac{q^{\frac{1}{2}i^2-ij+j^2-\frac{1}{2}i}}{(q;q)_i(q^2;q^2)_j}
            &=(-1;q)_\infty^2 , \label{t1-2-3}\\
            \sum_{i,j\ge 0}\frac{q^{i^2-2ij+2j^2-i+j}}{(q^2;q^2)_i(q^4;q^4)_j}
            &=(-1,-q;q^2)_\infty ,\label{t1-2-4} \\
            \sum_{i,j\ge 0}\frac{q^{i^2-2ij+2j^2+2j}}{(q^2;q^2)_i(q^4;q^4)_j}
            &=            \frac{(q^2;q^2)_\infty^2(q^6;q^6)_\infty^2}{(q;q)_\infty(q^3;q^3)_\infty(q^4;q^4)_\infty^2} . \label{t1-2-5}
        \end{align}
The identities  \eqref{t1-2-1} and \eqref{t1-2-5} appear in the work of Li and Wang as \cite[Eqs.\ (3.54) and (3.55)]{LW}. The remaining three identities follow from the following identity \cite[Eq.\ (3.53)]{LW}:
\begin{align}\label{eq-Li-Wang}
\sum_{i,j\geq 0}\frac{q^{i^2-2ij+2j^2-i}u^j}{(q^2;q^2)_i(q^4;q^4)_j}=(-1,-u;q^2)_\infty.
\end{align}
Interestingly, with $u=q^a$ ($a\in \mathbb{Q})$ the Nahm sum in the left side corresponds to 
\begin{align}\label{data-dual-exam1-Cao-Wang}
    A=\begin{pmatrix}
        1 & -1/2 \\ -1 & 1
    \end{pmatrix}, \quad b=\begin{pmatrix}
        -1/2 \\ a
    \end{pmatrix}, \quad d=(1,2),
\end{align}
which is exactly the dual of \eqref{data-exam1-Cao-Wang}.

We find new identities for the following choices of $(A,b,d)$ where $A$ and $d$ are the same with Example 1:
\begin{align}\label{new-exam-1}
             A = \begin{pmatrix}
        2 & 1\\
        2 & 2
\end{pmatrix}, \quad b=\begin{pmatrix}
    (a-3)/2 \\ a
\end{pmatrix}, \quad d=(1,2), \quad a\in \mathbb{Q}.
\end{align}
This corresponds to the special case $u=q^a$ of the following identity. 
\begin{theorem}\label{thm-new-exam1}
We have
\begin{align}
\sum_{i,j\ge 0}\frac{q^{2i^2+4ij+4j^2-3i}u^{i+2j}}{(q^2;q^2)_i(q^4;q^4)_j}
=(1+uq+uq^{-1})(-uq^3;q^2)_\infty.
\end{align}
\end{theorem}
\begin{proof}
By \eqref{euler} we have
\begin{align}
        &\sum_{i,j\ge 0}\frac{x^{i+2j}q^{i^2-3i}}{(q^2;q^2)_i(q^4;q^4)_j}
        =\frac{(-xq^{-2};q^2)_\infty}{(x^2;q^4)_\infty}
        =\frac{1+xq^{-2}}{(x;q^2)_\infty}\notag
        \\
        &=(1+xq^{-2})\sum_{m\ge 0}\frac{x^m}{(q^2;q^2)_m}
        =1+\sum_{m\ge 1}\left(\frac{1}{(q^2;q^2)_m}+\frac{1}{q^2(q^2;q^2)_{m-1}}\right)x^m.
\end{align}
Comparing the coefficients of $x^m$ on both sides, we deduce that
 \begin{align}
 \sum_{i+2j=m}\frac{q^{i^2-3i}}{(q^2;q^2)_i(q^4;q^4)_j}
 =\frac{1}{(q^2;q^2)_m}+\frac{1}{q^2(q^2;q^2)_{m-1}}
    \end{align}   
where we used the convention that $1/(q;q)_n=0$ for negative $n$.
We have
\begin{align*}
        &\sum_{i,j\ge 0}\frac{q^{2i^2+4ij+4j^2-3i}u^{i+2j}}{(q^2;q^2)_i(q^4;q^4)_j}
        =\sum_{m\ge 0}q^{m^2}u^m
        \sum_{i+2j=m}\frac{q^{i^2-3i}}{(q^2;q^2)_i(q^4;q^4)_j}\notag
        \\
        &=1+\sum_{m\ge 1}q^{m^2}u^m\left(\frac{1}{(q^2;q^2)_m}+\frac{1}{q^2(q^2;q^2)_{m-1}}\right)\notag
        \\
        &=\sum_{m\ge 0}\frac{q^{m^2}u^m}{(q^2;q^2)_m}+
        \sum_{m\ge 0}\frac{q^{(m+1)^2-2}u^{m+1}}{(q^2;q^2)_m}\notag
        \\
        &=
        (-uq;q^2)_\infty+ uq^{-1}(-uq^3;q^2)_\infty 
        =(1+uq+uq^{-1})(-uq^3;q^2)_\infty. \qedhere
\end{align*} 
\end{proof}
The following case is dual to \eqref{new-exam-1}, and they share the same $A,d$ with Example 2:
\begin{align}\label{new-exam-1-dual}
    A=\begin{pmatrix}
            1 & -1/2\\
            -1 & 1
        \end{pmatrix}, \quad b=\begin{pmatrix} -3/2 \\ (a+3)/2 \end{pmatrix}, \quad d=(1,2), \quad a\in \mathbb{Q}.
\end{align}
This corresponds to the following identity with $v=q^a$.
\begin{theorem}\label{thm-new-exam2}
We have
\begin{align}
\sum_{i,j\geq 0} \frac{q^{i^2-2ij+2j^2-3i+3j}v^j}{(q^2;q^2)_i(q^4;q^4)_j}=2q^{-2}(1+qv+q^2)(-q^2,-q^3v;q^2)_\infty.
\end{align}
\end{theorem}
\begin{proof}
By \eqref{euler} we have
\begin{align*}
        &\sum_{i,j\ge 0}\frac{q^{i^2-2ij+2j^2-3i+3j}v^j}{(q^2;q^2)_i(q^4;q^4)_j}
        =\sum_{j\ge 0}\frac{q^{2j^2+3j}v^j}{(q^4;q^4)_j}
        \sum_{i\ge 0}\frac{q^{i^2-i}(q^{-2-2j})^i}{(q^2;q^2)_i}\notag
        \\
        &=\sum_{j\ge 0}\frac{q^{2j^2+3j}v^j(-q^{-2-2j};q^2)_\infty}{(q^4;q^4)_j}
        =(-q^{-2};q^2)_\infty\sum_{j\ge 0}\frac{q^{2j^2+3j}v^j(-q^{-2-2j};q^2)_j}{(q^4;q^4)_j}\notag
        \\
        &=(-q^{-2};q^2)_\infty\sum_{j\ge 0}\frac{q^{j^2}v^j(-q^4;q^2)_j}{(q^4;q^4)_j}
        =(-q^{-2};q^2)_\infty\sum_{j\ge 0}\frac{q^{j^2}v^j(1+q^{2j+2})}{(q^2;q^2)_j(1+q^2)}\notag
        \\
        &=2q^{-2}(-q^2;q^2)_\infty
        \left(\sum_{j\ge 0}\frac{q^{j^2}v^j}{(q^2;q^2)_j} +q^2\sum_{j\ge 0}\frac{q^{j^2+2j}v^j}{(q^2;q^2)_j}\right)\notag
        \\
        &=2q^{-2}(-q^2;q^2)_\infty((-qv;q^2)_\infty+q^2(-q^3v;q^2)_\infty)\\
&=2q^{-2}(1+qv+q^2)(-q^2,-q^3v;q^2)_\infty. \qedhere
    \end{align*}
\end{proof}




  
\subsection{Examples 3 and 4}
Example 3 corresponds to
    \begin{align}\label{exam-2}
        A= \begin{pmatrix}
            1 & 1/2\\
            1 & 1
        \end{pmatrix}, \quad
        b \in \bigg\{
        \begin{pmatrix}
            0 \\ 0
        \end{pmatrix},
        \begin{pmatrix}
            0 \\ 1
        \end{pmatrix},
        \begin{pmatrix}
            1 \\ 1
        \end{pmatrix}
        \bigg\}, \quad
        d=(1,2) . 
    \end{align}
Recall the functions $F_\sigma(u,v;q)$ defined in \eqref{Fc-defn}. Let
\begin{align}
&F(u,v;q):=F_0(u,v;q)+F_1(u,v;q)=\sum_{n_1,n_2\geq 0} \frac{q^{\frac{1}{2}n_1^2+n_1n_2+n_2^2}u^{n_1}v^{n_2}}{(q;q)_{n_1}(q^2;q^2)_{n_2}}. \label{exam2-F-defn}
\end{align}
The modularity of this example follows from the identities  \cite[Eqs.\ (3.42), (3.43) and (3.44)]{LW}:
        \begin{align}
            \sum_{i,j\ge 0}\frac{q^{i^2+2ij+2j^2}}{(q^2;q^2)_i(q^4;q^4)_j}
            &=\frac{(-q;q^2)_\infty(q^3,q^4,q^7;q^7)_\infty}{(q^2;q^2)_\infty} , \label{t1-3-1} 
            \\
            \sum_{i,j\ge 0}\frac{q^{i^2+2ij+2j^2+2j}}{(q^2;q^2)_i(q^4;q^4)_j}
            &=
            \frac{(-q;q^2)_\infty(q^2,q^5,q^7;q^7)}{(q^2;q^2)_\infty} , \label{t1-3-2}
            \\
            \sum_{i,j\ge 0}\frac{q^{i^2+2ij+2j^2+2i+2j}}{(q^2;q^2)_i(q^4;q^4)_j}&=
            \frac{(-q;q^2)_\infty(q,q^6,q^7;q^7)_\infty}{(q^2;q^2)_\infty} . \label{t1-3-3}
        \end{align}

Now we give a proof for Theorem \ref{thm-parity}.
\begin{proof}[Proof of Theorem \ref{thm-parity}]
For $\sigma \in\{0,1\}$, we have
\begin{align}
&F_\sigma(u,v;q)=\sum_{\begin{smallmatrix}
n_1\equiv \sigma \!\! \pmod{2} \\ n_1\geq 0 
\end{smallmatrix}} \frac{q^{\frac{1}{2}n_1^2}u^{n_1}}{(q;q)_{n_1}}(-vq^{n_1+1};q^2)_\infty \nonumber \\
&=\sum_{k\geq 0} \frac{q^{\frac{1}{2}(2k+\sigma)^2}u^{2k+\sigma}}{(q;q)_{2k+\sigma}}(-vq^{2k+\sigma+1};q^2)_\infty. \label{Fc-start}
\end{align}

We have
\begin{align*}
&F_0(1,1;q)=\sum_{k=0}^\infty  \frac{q^{2k^2}}{(q;q)_{2k}}(-q^{2k+1};q^2)_\infty =(-q;q^2)_\infty \sum_{k=0}^\infty \frac{q^{2k^2}}{(q;q)_{2k}(-q;q^2)_k} \nonumber \\
&=(-q;q^2)_\infty \sum_{k=0}^\infty \frac{q^{2k^2}}{(q^2;q^4)_k(q^2;q^2)_k}.
\end{align*}
By \eqref{S61} we obtain \eqref{r1}.

We have
\begin{align*}
&F_1(1,1;q)=q^{\frac{1}{2}}\sum_{k=0}^\infty \frac{q^{2k^2+2k}}{(q;q)_{2k+1}}(-q^{2k+2};q^2)_\infty =
q^{\frac{1}{2}}(-q^2;q^2)_\infty \sum_{k=0}^\infty \frac{q^{2k^2+2k}}{(q;q)_{2k+1}(-q^2;q^2)_k} \nonumber \\
&=q^{\frac{1}{2}}(-q^2;q^2)_\infty \sum_{k=0}^\infty \frac{q^{2k^2+2k}}{(q;q^2)_{k+1}(q^4;q^4)_k}.
\end{align*}
By \eqref{S31} we obtain \eqref{r2}.

We have
\begin{align*}
&F_0(1,q;q)=\sum_{k=0}^\infty \frac{q^{2k^2}}{(q;q)_{2k}}(-q^{2k+2};q^2)_\infty
=(-q^2;q^2)_\infty \sum_{k=0}^\infty \frac{q^{2k^2}}{(q;q^2)_k(q^4;q^4)_k}.
\end{align*}
By \eqref{S33} we obtain \eqref{r3}.

We have 
\begin{align*}
&F_1(1,q;q)=q^{\frac{1}{2}}\sum_{k=0}^\infty \frac{q^{2k^2+2k}}{(q;q)_{2k+1}}(-q^{2k+3};q^2)_\infty
=q^{\frac{1}{2}}(-q;q^2)_\infty \sum_{k=0}^\infty \frac{q^{2k^2+2k}}{(q;q)_{2k+1}(-q;q^2)_{k+1}} \nonumber \\
&=q^{\frac{1}{2}}(-q;q^2)_\infty \sum_{k=0}^\infty \frac{q^{2k^2+2k}}{(q^2;q^2)_k(q^2;q^4)_{k+1}}.
\end{align*}
By \eqref{S60} we obtain \eqref{r4}.

We have
\begin{align*}
&F_0(q,q;q)=\sum_{k=0}^\infty \frac{q^{2k^2+2k}}{(q;q)_{2k}}(-q^{2k+2};q^2)_\infty 
=(-q^2;q^2)_\infty \sum_{k=0}^\infty \frac{q^{2k^2+2k}}{(q;q)_{2k}(-q^2;q^2)_k} \nonumber \\
&=(-q^2;q^2)_\infty \sum_{k=0}^\infty \frac{q^{2k^2+2k}}{(q,-q^2;q^2)_k(q^2;q^2)_k}.
\end{align*}
By \eqref{S32} we obtain \eqref{r5}.

We have
\begin{align*}
&F_1(q,q;q)=q^{\frac{3}{2}}\sum_{k=0}^\infty \frac{q^{2k^2+4k}}{(q;q)_{2k+1}}(-q^{2k+3};q^2)_\infty =q^{\frac{3}{2}}(-q;q^2)_\infty \sum_{k=0}^\infty \frac{q^{2k^2+4k}}{(q;q)_{2k+1}(-q;q^2)_{k+1}} \nonumber \\
&=q^{\frac{3}{2}}(-q;q^2)_\infty \sum_{k=0}^\infty \frac{q^{2k^2+4k}}{(q^2;q^2)_k(q^2;q^4)_{k+1}}.
\end{align*}
By \eqref{S59} we obtain \eqref{r6}.
\end{proof}

To prove Mizuno's conjectural formula \eqref{eq-conj-tran}, we first study the modular transformation formulas satisfied by the vector-valued function
\begin{align}
 V(\tau):=&\Big(q^{-3/56}F_0(1,1;-q),-\zeta_4q^{1/56}F_1(1,-q;-q),\zeta_4q^{9/56}F_1(-q,-q;-q),\nonumber \\
    &\quad -\zeta_4q^{-3/56}F_1(1,1;-q),q^{1/56}F_0(1,-q;-q),q^{9/56}F_0(-q,-q;-q)\Big)^\mathrm{T}.
\end{align} 
Note that except for factors of roots of unity, $V(\tau)$ is essentially obtained from $U(\tau)$ by replacing $q$ by $-q$. The reordering will make the transformation formulas look simpler.
We will show that $U(\tau)$ and $V(\tau)$ are vector-valued modular functions on $\Gamma_0(2)$ and $\Gamma_0(4)$, respectively.

\begin{theorem}\label{thm-vector}
We have
\begin{align}\label{V-tran}
   & V(\tau+1)=\Lambda^4 V(\tau), \quad   V\left(-\frac{1}{4\tau}\right)=\begin{pmatrix}
        0 & W \\ W^\mathrm{T} & 0 
    \end{pmatrix}V(\tau), \\
&V\left(\frac{\tau}{4\tau+1} \right)=\begin{pmatrix}
        WP^{-4}W^\mathrm{T} & 0 \\ 0 & W^\mathrm{T} P^{-4}W 
    \end{pmatrix}V(\tau). \label{V-tran-third}
\end{align}
where $M$ is given in \eqref{eq-M-defn}, $P, \Lambda$ and $W$ are given in Theorem \ref{thm-conj}. 
As a consequence, $(V_1(\tau),V_2(\tau),V_3(\tau))^\mathrm{T}$, $(V_4(\tau),V_5(\tau),V_6(\tau))^\mathrm{T}$ and $V(\tau)$ are vector-valued modular functions on $\Gamma_0(4)$.
\end{theorem}
\begin{proof}[Proof of Theorem \ref{thm-vector}]
For convenience, we denote the $i$-th component of $V(\tau)$ by $V_i(\tau)$. By Theorem \ref{thm-parity} and \eqref{Jacobi} we have
\begin{align}
V_1(\tau)=q^{-\frac{3}{56}}\frac{(q;q^2)_\infty}{(q^2;q^2)_\infty}\sum_{n=-\infty}^\infty (-1)^nq^{14n^2+2n}=\frac{\mathfrak{f}_1(2\tau)}{\eta(2\tau)}g_{1,7}(2\tau),  \label{V1-g}\\
V_2(\tau)=q^{\frac{29}{56}}\frac{(q;q^2)_\infty}{(q^2;q^2)_\infty} \sum_{n=-\infty}^\infty (-1)^nq^{14n^2+6n}=\frac{\mathfrak{f}_1(2\tau)}{\eta(2\tau)}g_{3,7}(2\tau), \label{V2-g} \\
V_3(\tau)=q^{\frac{93}{56}}\frac{(q;q^2)_\infty}{(q^2;q^2)_\infty}\sum_{n=-\infty}^\infty (-1)^nq^{14n^2+10n} =\frac{\mathfrak{f}_1(2\tau)}{\eta(2\tau)} g_{5,7}(2\tau),  \label{V3-g} \\
V_4(\tau)=q^{\frac{25}{56}}\frac{(-q^2;q^2)_\infty}{(q^2;q^2)_\infty}\sum_{n=-\infty}^\infty (-1)^nq^{\frac{7}{2}n^2+\frac{5}{2}n}=\frac{\mathfrak{f}_2(2\tau)}{\eta(2\tau)}g_{5,7}\left(\frac{\tau}{2}\right), \label{V4-g} \\
V_5(\tau)=q^{\frac{1}{56}}\frac{(-q^2;q^2)_\infty}{(q^2;q^2)_\infty}\sum_{n=-\infty}^\infty (-1)^nq^{\frac{7}{2}n^2+\frac{1}{2}n}=\frac{\mathfrak{f}_2(2\tau)}{\eta(2\tau)}g_{1,7}\left(\frac{\tau}{2}\right),  \label{V5-g} \\
V_6(\tau)=q^{\frac{9}{56}}\frac{(-q^2;q^2)_\infty}{(q^2;q^2)_\infty}\sum_{n=-\infty}^\infty (-1)^nq^{\frac{7}{2}n^2+\frac{3}{2}n}=\frac{\mathfrak{f}_2(2\tau)}{\eta(2\tau)}g_{3,7}\left(\frac{\tau}{2}\right). \label{V6-g}
\end{align}
Using \eqref{eta-tran}, \eqref{Weber-2} and \eqref{fg-translation}, it is easy to prove the first transformation formula in \eqref{V-tran}.

Applying Lemma \ref{lem-modular} and \eqref{Weber-1}, we deduce that for $j\in \{1,2,3\}$,
\begin{align}
&V_j\left(-\frac{1}{4\tau}\right)=\frac{\mathfrak{f}_1(-1/(2\tau))}{\eta(-1/(2\tau))}g_{2j-1,7}\left(-\frac{1}{2\tau}\right)=\frac{1}{\sqrt{7}}\frac{\mathfrak{f}_2(2\tau)}{\eta(2\tau)}\sum_{k=0}^{13} e^{\frac{\pi i(2k+1)(2j-1)}{14}}h_{k+\frac{1}{2},7}(2\tau) \nonumber \\
&=\frac{1}{\sqrt{7}}\frac{\mathfrak{f}_2(2\tau)}{\eta(2\tau)}\sum_{k=0}^{6} \left( e^{\frac{\pi i(2k+1)(2j-1)}{14}} h_{k+\frac{1}{2},7}(2\tau)+e^{\frac{\pi i(2(13-k)+1)(2j-1)}{14}}h_{13-k+\frac{1}{2},7}(2\tau) \right)\nonumber \\
&=\frac{1}{\sqrt{7}}\frac{\mathfrak{f}_2(2\tau)}{\eta(2\tau)}\sum_{k=0}^{6} 2\cos \frac{(2k+1)(2j-1)\pi}{14} h_{k+\frac{1}{2},7}(2\tau) \quad \text{(by \eqref{g-h-period})} \nonumber \\
&=\frac{2}{\sqrt{7}}\frac{\mathfrak{f}_2(2\tau)}{\eta(2\tau)}\sum_{k=0}^{2} \cos \frac{(2k+1)(2j-1)\pi}{14} \left(h_{k+\frac{1}{2},7}(2\tau)-h_{\frac{13}{2}-k,7}(2\tau)  \right) \nonumber \\
&=\frac{2}{\sqrt{7}}\frac{\mathfrak{f}_2(2\tau)}{\eta(2\tau)}\sum_{k=0}^{2} \cos \frac{(2k+1)(2j-1)\pi}{14} \left(h_{4k+2,28}\left(\frac{\tau}{2}\right)-h_{26-4k,28}\left(\frac{\tau}{2}\right)  \right) \nonumber \\
&=\frac{2}{\sqrt{7}}\frac{\mathfrak{f}_2(2\tau)}{\eta(2\tau)}\sum_{k=0}^{2} \cos \frac{(2k+1)(2j-1)\pi}{14} g_{2k+1,7}\left(\frac{\tau}{2}\right). \quad \text{(by \eqref{g-h-change})} 
 \label{V1-tran}
\end{align}
Setting $j=1,2,3$ we obtain
\begin{align}
V_1\left(-\frac{1}{4\tau}\right)&=\frac{2}{\sqrt{7}}\left(\sin \frac{\pi}{7} V_4(\tau)+\sin \frac{3\pi}{7} V_5(\tau)+\sin \frac{2\pi}{7} V_6(\tau)\right), \nonumber \\
V_2\left(-\frac{1}{4\tau}\right)&=\frac{2}{\sqrt{7}}\left(-\sin \frac{3\pi}{7} V_4(\tau)+\sin \frac{2\pi}{7} V_5(\tau)-\sin \frac{\pi}{7} V_6(\tau)\right),  \\
V_3\left(-\frac{1}{4\tau}\right)&=\frac{2}{\sqrt{7}}\left(\sin \frac{2\pi}{7} V_4(\tau)+\sin \frac{\pi}{7} V_5(\tau)-\sin \frac{3\pi}{7} V_6(\tau)\right). \nonumber
\end{align}
Similarly, we can deduce that
\begin{align}
V_4\left(-\frac{1}{4\tau}\right)&=\frac{2}{\sqrt{7}}\left(\sin \frac{\pi}{7} V_1(\tau)-\sin \frac{3\pi}{7} V_2(\tau)+\sin \frac{2\pi}{7} V_3(\tau)\right), \label{V4-tran} \\
V_5\left(-\frac{1}{4\tau}\right)&=\frac{2}{\sqrt{7}}\left(\sin \frac{3\pi}{7} V_1(\tau)+\sin \frac{2\pi}{7} V_2(\tau)+\sin \frac{\pi}{7} V_3(\tau)\right), \nonumber \\
V_6\left(-\frac{1}{4\tau}\right)&=\frac{2}{\sqrt{7}}\left(\sin \frac{2\pi}{7} V_1(\tau)-\sin \frac{\pi}{7} V_2(\tau)-\sin \frac{3\pi}{7} V_3(\tau)\right). \nonumber
\end{align}
This proves the second transformation formula. 

We denote $H=\left(\begin{smallmatrix}
    0 & W \\ W^\mathrm{T} & 0
\end{smallmatrix}\right)$. By the second formula in \eqref{V-tran} we have
\begin{align*}
&V\left(\frac{\tau}{4\tau+1} \right)=HV\left(-\frac{4\tau+1}{4\tau} \right)=HV\left(-1-\frac{1}{4\tau}\right) \\
&=H\Lambda^{-4}V\left(-\frac{1}{4\tau}\right)=H\Lambda^{-4}HV(\tau).
\end{align*}

Since  $\Gamma_0(4)$ is generated by $\pm\left(\begin{smallmatrix}
    1 & 1 \\ 0  &1 
\end{smallmatrix}\right)$ and $\pm\left(\begin{smallmatrix}
    1 & 0 \\ 4 & 1
\end{smallmatrix}\right)$, the last assertion follows. 
\end{proof}

Now we are able to prove Theorem \ref{thm-conj}.
\begin{proof}[Proof of Theorem \ref{thm-conj}]
We denote the $i$-th component of $U(\tau)$ as $U_i(\tau)$. Note that replacing $q$ by $-q$ corresponds to replacing $\tau$ by $\tau+\frac{1}{2}$. Therefore, we have
\begin{align}\label{U-V-relation}
    U(\tau)=\Lambda^{-1}V\left(\tau+\frac{1}{2}\right).
\end{align}
It follows from \eqref{V1-g}--\eqref{V3-g} that
\begin{align}
    U_1(\tau)=\frac{\mathfrak{f}(2\tau)}{\eta(2\tau)}g_{1,7}(2\tau), \quad U_2(\tau)=\frac{\mathfrak{f}(2\tau)}{\eta(2\tau)}g_{3,7}(2\tau), \quad 
    U_3(\tau)=\frac{\mathfrak{f}(2\tau)}{\eta(2\tau)}g_{5,7}(2\tau). 
\end{align}

Conjecture \ref{conj-Mizuno} is equivalent to 
\begin{align}\label{U-M-transform}
\begin{pmatrix}
U_1 \\ U_2 \\ U_3
\end{pmatrix}\left(-\frac{1}{4\tau}\right)=M\begin{pmatrix}
U_1+U_4 \\ U_2+U_5 \\ U_3+U_6
\end{pmatrix}(2\tau),  \quad \begin{pmatrix}
U_4 \\ U_5 \\ U_6
\end{pmatrix}\left(-\frac{1}{4\tau}\right)=M\begin{pmatrix}
U_1-U_4 \\ U_2-U_5 \\ U_3-U_6
\end{pmatrix}(2\tau).
\end{align}
Arguing similarly as in \eqref{V1-tran}, we have
\begin{align}
&U_1\left(-\frac{1}{4\tau}\right)=\frac{\mathfrak{f}(2\tau)}{\eta(2\tau)}\left(\alpha_3 g_{1,7}\left(\frac{\tau}{2}\right)+\alpha_2 g_{3,7}\left(\frac{\tau}{2}\right)+\alpha_1g_{5,7}\left(\frac{\tau}{2}\right)\right), \label{U1-tran} \\
&U_2\left(-\frac{1}{4\tau}\right)=\frac{\mathfrak{f}(2\tau)}{\eta(2\tau)}\left(\alpha_2g_{1,7}(\frac{\tau}{2})-\alpha_1 g_{3,7}\left(\frac{\tau}{2}\right)-\alpha_3g_{5,7}\left(\frac{\tau}{2}\right)\right), \label{U2-tran} \\
&U_3\left(-\frac{1}{4\tau}\right)=\frac{\mathfrak{f}(2\tau)}{\eta(2\tau)}\left(\alpha_1 g_{1,7}\left(\frac{\tau}{2}\right)-\alpha_3 g_{3,7}\left(\frac{\tau}{2}\right)+\alpha_2g_{5,7}\left(\frac{\tau}{2}\right)\right). \label{U3-tran} 
\end{align}
We claim that
\begin{align}
\frac{\mathfrak{f}(2\tau)}{\eta(2\tau)}g_{1,7}(\frac{\tau}{2})=U_1(2\tau)+U_4(2\tau), \label{U-sum-1}\\
\frac{\mathfrak{f}(2\tau)}{\eta(2\tau)}g_{3,7}(\frac{\tau}{2})=U_2(2\tau)+U_5(2\tau),  \label{U-sum-2} \\
\frac{\mathfrak{f}(2\tau)}{\eta(2\tau)}g_{5,7}(\frac{\tau}{2})=U_3(2\tau)+U_6(2\tau). \label{U-sum-3}
\end{align}
In fact, by definition we have
\begin{align*}
&U_1(2\tau)+U_4(2\tau)=q^{-3/28}(F_0(1,1;q^2)+F_1(1,1;q^2))=q^{-3/28}F(1,1;q^2), \\
&U_2(2\tau)+U_5(2\tau)=q^{1/28}(F_0(1,q^2;q^2)+F_1(1,q^2;q^2))=q^{1/28}F(1,q^2;q^2),\\
&U_3(2\tau)+U_6(2\tau)=q^{9/28}(F_0(q^2,q^2;q^2)+F_1(q^2,q^2;q^2))=q^{9/28}F(q^2,q^2;q^2).
\end{align*}
By \eqref{t1-3-1}--\eqref{t1-3-3}, we see that the right sides of the above identities are exactly the left sides of \eqref{U-sum-1}, \eqref{U-sum-2} and \eqref{U-sum-3}, respectively. Hence \eqref{U-sum-1}--\eqref{U-sum-3} hold and we prove the first transformation formula in \eqref{U-M-transform}.

Next, replacing $\tau$ by $-1/(8\tau)$ in \eqref{U-sum-1}--\eqref{U-sum-3}, arguing similarly as in \eqref{V1-tran}, we deduce that
\begin{align}
&U_1\left(-\frac{1}{4\tau}\right)+U_4\left(-\frac{1}{4\tau}\right)=\frac{\mathfrak{f}(-1/4\tau)}{\eta(-1/4\tau)}g_{1,7}\left(-\frac{1}{16\tau}\right)  \nonumber \\
&=2\Big(\alpha_3 U_1(2\tau)+\alpha_2 U_2(2\tau)+\alpha_1 U_3(2\tau)\Big), \label{U1-U4-tran} \\
&U_2\left(-\frac{1}{4\tau}\right)+U_5\left(-\frac{1}{4\tau}\right)=\frac{\mathfrak{f}(-1/4\tau)}{\eta(-1/4\tau)}g_{3,7}\left(-\frac{1}{16\tau}\right)  \nonumber \\
&=2\Big(\alpha_2 U_1(2\tau)-\alpha_1 U_2(2\tau)-\alpha_3 U_3(2\tau)\Big), \label{U2-U5-tran} \\
&U_3\left(-\frac{1}{4\tau}\right)+U_6\left(-\frac{1}{4\tau}\right)=\frac{\mathfrak{f}(-1/4\tau)}{\eta(-1/4\tau)}g_{5,7}\left(-\frac{1}{16\tau}\right)  \nonumber \\
&=2\Big(\alpha_1 U_1(2\tau)-\alpha_3 U_2(2\tau)+\alpha_2 U_3(2\tau)\Big). \label{U3-U6-tran}
\end{align}
This implies
\begin{align}\label{U1-U4-matrix}
    \begin{pmatrix}
        U_1+U_4 \\ U_2+U_5 \\ U_3+U_6
    \end{pmatrix}\left(-\frac{1}{4\tau}\right)=2M\begin{pmatrix}
        U_1 \\ U_2 \\ U_3
    \end{pmatrix}(2\tau). 
\end{align}
Subtracting it by the first formula in \eqref{U-M-transform}, we obtain the second transformation formula in \eqref{U-M-transform}.

By Theorem \ref{thm-vector} and \eqref{U-V-relation} we have
\begin{align}
    U(\tau+1)=\Lambda^{-1}V\left(\tau+\frac{3}{2}\right)=\Lambda^3 V\left(\tau+\frac{1}{2}\right)=\Lambda^4 U(\tau).
\end{align}
Next, 
\begin{align*}
&U\left(\frac{\tau}{2\tau+1}\right)=U\left(-\frac{1}{2(2\tau+1)}+\frac{1}{2}\right)=\Lambda^{-1}V\left( -\frac{1}{2(2\tau+1)}+1 \right) \nonumber \\
&=\Lambda^3 V\left( -\frac{1}{2(2\tau+1)} \right)=\Lambda^3 \begin{pmatrix}
    0 & W \\ W^{\mathrm{T}} & 0
\end{pmatrix} V\left(\tau+\frac{1}{2}\right)=\Lambda^3 \begin{pmatrix}
    0 & W \\ W^{\mathrm{T}} & 0 
\end{pmatrix}\Lambda U(\tau) \nonumber \\
&=\begin{pmatrix}
    0 & P^3 WP \\ P^3 W^{\mathrm{T}} P & 0
\end{pmatrix}U(\tau).
\end{align*}
Since $\Gamma_0(2)$ is generated by $\left(\begin{smallmatrix}
    1 &  1  \\ 0 &1
\end{smallmatrix}\right)$ and $\left(\begin{smallmatrix}
 1 & 0 \\ 2 & 1   
\end{smallmatrix}\right)$, we know that $U(\tau)$ is a vector-valued modular function on $\Gamma_0(2)$.
\end{proof}


Example 4 (the dual of Example 3) corresponds to
\begin{align}
    A=\begin{pmatrix}
        2 & -1 \\ -2 & 2
    \end{pmatrix} , \quad
    b \in \bigg\{ 
    \begin{pmatrix}
        0 \\0
    \end{pmatrix} , 
    \begin{pmatrix}
        -1 \\ 2
    \end{pmatrix} ,
    \begin{pmatrix}
        1 \\ 0
    \end{pmatrix} 
    \bigg \} , \quad
    d=(1,2) .  
\end{align}
We establish the following identities to prove its modularity.
\begin{theorem}
We have
\begin{align}
    &\sum_{i,j\ge 0}\frac{q^{i^2-2ij+2j^2}}{(q;q)_i(q^2;q^2)_j}=\frac{J_2^6J_{28}^3}{J_1^4J_4^2J_{4,28}J_{6,28}J_{8,28}}-2q\frac{J_4^2J_{4,28}J_{5,14}}{J_1^2J_2J_{28}} \label{12-1-a} \\
    &= 
     \frac{J_4^5J_{28}J_{6,56}J_{16,56}J_{22,56}}{J_{2}^4J_8^2J_{56}^{3}}
      +2q\frac{J_4J_8J_{56}^3}{J_{2,4}J_{4,8}J_{4,56}J_{16,56}J_{24,56}} , \label{12-1-b} \\
    &\sum_{i,j\ge 0}\frac{q^{i^2-2ij+2j^2-i+2j}}{(q;q)_i(q^2;q^2)_j}=2\frac{J_4^2J_{1,14}J_{12,28}}{J_1^2J_2J_{28}}-q\frac{J_2^6J_{28}^3}{J_1^4J_4^2J_{4,28}J_{10,28}J_{12,28}} \label{12-2-a} 
    \\
   & =2\frac{J_8J_{56}J_{24,56}}{J_2^2J_{12,56}}+q\frac{J_4^5J_{28}
   J_{8,56}J_{10,56}J_{18,56}}{J_{2}^4J_{8}^2J_{56}^{3}} ,
    \label{12-2-b} \\
    &\sum_{i,j\ge 0}\frac{q^{i^2-2ij+2j^2+i}}{(q;q)_i(q^2;q^2)_j}=2\frac{J_4^3J_{14}J_{3,28}J_{11,28}}{J_1^2J_2J_{28}J_{4,28}J_{12,28}}-\frac{J_2^6J_{28}^3}{J_1^4J_4^2J_{2,28}J_{8,28}J_{12,28}} ,  \label{12-3-a} \\
    &=  
    \frac{J_{4}^5J_{28}J_{2,56}J_{24,56}J_{26,56}}{J_2^4J_8^2J_{56}^3}
    +2q^3\frac{J_8J_{56}J_{16,56}}{J_2^2J_{20,56}}. \label{12-3-b}
\end{align}
\end{theorem}
\begin{proof}
Since the equivalence between the first and the second expressions in each identity can be automatically proved using the method in \cite{Frye-Garvan}, we only prove the first expression in each identity.

Let
\begin{align}
F(u,v)=F(u,v;q):=\sum_{i,j\geq 0} \frac{q^{(i-j)^2+j^2}u^iv^j}{(q;q)_i(q^2;q^2)_j}.
\end{align}
We have by \eqref{euler} and \eqref{Jacobi} that
\begin{align}
    &F(u,v)=\oint \sum_{i=0}^\infty \frac{u^iz^i}{(q;q)_i} \sum_{j=0}^\infty \frac{q^{j^2}v^jz^{-j}}{(q^2;q^2)_j} \sum_{k=-\infty}^\infty q^{k^2}z^{-k} \frac{dz}{2\pi iz} \nonumber \\
    &=\oint \frac{(-qv/z,-qz,-q/z,q^2;q^2)_\infty}{(uz;q)_\infty} \frac{dz}{2\pi iz}. \label{12-F-int}
\end{align}

(1) By \eqref{12-F-int} we have
\begin{align}
F(1,1)=\oint \frac{(-qz,-q/z,-q/z,q^2;q^2)_\infty}{(z,qz;q^2)_\infty} \frac{dz}{2\pi iz} =S_1(q)+S_2(q), \label{12-F-start}
\end{align}
where by Lemma \ref{lem-integral} we have
\begin{align}
    &S_1(q)=\frac{(-q,-q,-q;q^2)_\infty}{(q;q^2)_\infty} \sum_{n=0}^\infty \frac{(-q;q^2)_n q^{n^2+n}}{(q^2,-q,-q,q;q^2)_n} \nonumber \\
    &=\frac{(-q;q^2)_\infty^3}{(q;q^2)_\infty} \sum_{n=0}^\infty \frac{q^{n^2+n}}{(q^2;q^2)_n(q^2;q^4)_n} =\frac{J_2^6J_{28}^3}{J_1^4J_4^2J_{4,28}J_{6,28}J_{8,28}} \quad \text{(by \eqref{S81})}\label{12-S1}
\end{align}
and 
\begin{align}
    &S_2(q)=\frac{(-q^2,-q^2,-1;q^2)_\infty}{(q^{-1};q^2)_\infty} \sum_{n=0}^\infty \frac{(-q^2;q^2)_nq^{n^2+2n}}{(q^2,-q^2,-q^2,q^3;q^2)_n} \nonumber \\
    &=-2q\frac{(-q^2;q^2)_\infty^3}{(q;q^2)_\infty} \sum_{n=0}^\infty \frac{q^{n^2+2n}}{(q;q^2)_{n+1}(q^4;q^4)_n}=-2q\frac{J_4^2J_{4,28}J_{5,14}}{J_1^2J_2J_{28}}. \quad \text{(by \eqref{S119})} \label{12-S2}
\end{align}
Substituting \eqref{12-S1} and \eqref{12-S2} into \eqref{12-F-start}, we obtain \eqref{12-1-a}.

(2) By \eqref{12-F-int} we have
\begin{align}
   & F(q^{-1},q^{2})=\oint \frac{(-q^3/z,-qz,-q/z,q^2;q^2)_\infty}{(q^{-1}z;q)_\infty} \frac{dz}{2\pi iz} \nonumber \\
   &=\oint \frac{(-qz,-q/z,-q^3/z,q^2;q^2)_\infty}{(q^{-1}z,z;q^2)_\infty} \frac{dz}{2\pi iz}.
\end{align}
By Lemma \ref{lem-integral} we have
\begin{align}
    F(q^{-1},q^{2})=S_1(q)+S_2(q), \label{12-2-start}
\end{align}
where 
\begin{align}
    &S_1(q)=\frac{(-1,-q^2,-q^2;q^2)_\infty}{(q;q^2)_\infty} \sum_{n=0}^\infty \frac{(-1;q^2)_nq^{n^2+2n}}{(q^2,-1,-q^2,q;q^2)_n}   \label{12-2-S1} \\
    &=2\frac{(-q^2;q^2)_\infty^3}{(q;q^2)_\infty} \sum_{n=0}^\infty \frac{q^{n^2+2n}}{(q^2,-q^2,q;q^2)_n} =2\frac{J_4^2J_{1,14}J_{12,28}}{J_1^2J_2J_{28}} \quad \text{(by \eqref{S118}})  \nonumber
\end{align}
and
\begin{align}
    &S_2(q)=\frac{(-q,-q^3,-q;q^2)_\infty}{(q^{-1};q^2)_\infty} \sum_{n=0}^\infty \frac{(-q;q^2)_nq^{n^2+3n}}{(q^2,-q,-q^3,q^3;q^2)_n} \label{12-2-S2} \\
    &=-q\frac{(-q;q^2)_\infty^3}{(q;q^2)_\infty} \sum_{n=0}^\infty \frac{q^{n^2+3n}}{(q^2;q^2)_n(q^2;q^4)_{n+1}} =-q\frac{J_2^6J_{28}^3}{J_1^4J_4^2J_{4,28}J_{10,28}J_{12,28}}.  \quad \text{(by \eqref{S82})}  \nonumber
\end{align}
Substituting \eqref{12-2-S1} and \eqref{12-2-S2} into \eqref{12-2-start}, we obtain \eqref{12-2-a}.

(3) By \eqref{12-F-int} and Lemma \ref{lem-integral} we have
\begin{align}\label{12-3-split}
    F(q,1)=\oint \frac{(-qz,-q/z,-q/z,q^2;q^2)_\infty}{(qz,q^2z;q^2)_\infty} \frac{dz}{2\pi iz}=S_1(q)+S_2(q),
\end{align}
where
\begin{align}
    &S_1(q)=\frac{(-q^2,-q^2,-1;q^2)_\infty}{(q;q^2)_\infty} \sum_{n=0}^\infty \frac{(-q^2;q^2)_nq^{n^2}}{(q^2,-q^2,-q^2,q;q^2)_n}  \label{12-3-S1} \\
    &=2\frac{(-q^2;q^2)_\infty^3}{(q;q^2)_\infty} \sum_{n=0}^\infty \frac{q^{n^2}}{(q^2,-q^2,q;q^2)_n} =2\frac{J_4^3J_{14}J_{3,28}J_{11,28}}{J_1^2J_2J_{28}J_{4,28}J_{12,28}}  \quad \text{(by \eqref{S117})}  \nonumber 
\end{align}
and 
\begin{align}
    &S_2(q)=\frac{(-q^3,-q^3,-q^{-1};q^2)_\infty}{(q^{-1};q^2)_\infty} \sum_{n=0}^\infty \frac{(-q^3;q^2)_nq^{n^2+n}}{(q^2,-q^3,-q^3,q^3;q^2)_n} \label{12-3-S2} \\
    &=-\frac{(-q;q^2)_\infty^3}{(q;q^2)_\infty} \sum_{n=0}^\infty \frac{q^{n^2+n}}{(q^2;q^2)_n(q^2;q^4)_{n+1}} =-\frac{J_2^6J_{28}^3}{J_1^4J_4^2J_{2,28}J_{8,28}J_{12,28}}. \quad \text{(by \eqref{S80})} \nonumber  
\end{align}
Substituting \eqref{12-3-S1} and \eqref{12-3-S2} into \eqref{12-3-split} we obtain \eqref{12-3-a}.
\end{proof}
\begin{rem}
The second expression in each identity gives a 2-dissection formula for the first expression. We can also obtain a 2-dissection formula in an elementary way. From \cite[Eqs.\ (2.21) and (2.22)]{Wang2020} we find
\begin{align}
    \frac{1}{J_1^2}&=\frac{J_8^5}{J_2^5J_{16}^2}+2q\frac{J_4^2J_{16}^2}{J_2^5J_8}, \label{J1-square} \\
    \frac{1}{J_1^4}&=\frac{J_4^{14}}{J_2^{14}J_8^4}+4q\frac{J_4^2J_8^4}{J_2^{10}}. \label{J1-four}
\end{align}
Substituting these formulas into \eqref{12-1-a}, \eqref{12-2-a} and \eqref{12-3-a}, and  then extracting the terms with even and odd powers of $q$, we obtain their 2-dissection formulas.
\end{rem}

\subsection{Examples 5 and 6}

Example 5 corresponds to
\begin{align}\label{exam-3}
    A=\begin{pmatrix}
        4 & 2 \\ 6 &4
    \end{pmatrix}, \quad b \in \left\{\begin{pmatrix}
        0 \\ 0
    \end{pmatrix}\right\}, \quad d=(1,3).
\end{align}
The modularity follows from Capparelli's identity \eqref{eq-Capparelli}.

Example 6 (the dual of Example 5) corresponds to
\begin{align}
    A=\begin{pmatrix}
        1 & -1/2 \\
        -3/2 & 1
    \end{pmatrix} , \quad
    b \in \bigg\{ 
    \begin{pmatrix}
        0 \\ 0
    \end{pmatrix} \bigg\} , \quad
    d=(1,3) . 
\end{align}
To prove its modularity, we establish the following theorem.
\begin{theorem}\label{thm-exam-6}
We have
\begin{align}
    \sum_{i,j\ge 0}\frac{q^{i^2-3ij+3j^2}}{(q^2;q^2)_i(q^6;q^6)_j}&=\frac{J_2^2J_6J_{24}}{J_1J_3J_4J_{4,24}}+2q\frac{J_4J_{24}^3J_{4,24}}{J_2J_{2,24}J_{6,24}^2J_{10,24}} \label{eq-exam6-2} \\
    &=
    \frac{J_{24}^6}{J_{4,24}^3J_{6,24}^3}+3q\frac{J_{24}^6J_{4,24}}{J_{2,24}^2J_{6,24}^3J_{10,24}^2}. \label{eq-exam6-1}
\end{align}
\end{theorem}
\begin{proof}
We have by \eqref{euler} that
\begin{align}
&\sum_{i,j\geq 0} \frac{q^{i^2-3ij+3j^2}}{(q^2;q^2)_i(q^6;q^6)_j}=\sum_{i=0}^\infty \frac{q^{i^2}}{(q^2;q^2)_i}\sum_{j=0}^\infty \frac{q^{3(j^2-j)}\cdot q^{(3-3i)j}}{(q^6;q^6)_j} \nonumber \\
&=\sum_{i=0}^\infty \frac{q^{i^2}}{(q^2;q^2)_i}(-q^{3-3i};q^6)_\infty=S_0(q)+S_1(q), \label{36-S}
\end{align}
where $S_0(q)$ and $S_1(q)$ correspond to the sums with even and odd values of $i$, respectively.
We have
\begin{align}
&S_0(q)=\sum_{i=0}^\infty \frac{q^{4i^2}}{(q^2;q^2)_{2i}}(-q^{3-6i};q^6)_\infty  \nonumber \\
&=(-q^3;q^6)_\infty \sum_{i=0}^\infty \frac{q^{i^2}(-q^3;q^6)_i}{(q^2;q^2)_{2i}} =\frac{J_2J_6^2J_{2,24}J_{10,24}}{J_1J_3J_{12}J_{24}J_{4,24}}. \quad \text{(by \eqref{Rama538})} \label{36-S1}
\end{align}
Similarly,
\begin{align}
&S_1(q)=\sum_{i=0}^\infty \frac{q^{4i^2+4i+1}}{(q^2;q^2)_{2i+1}}(-q^{-6i};q^6)_\infty  \nonumber \\
&=2q(-q^6;q^6)_\infty \sum_{i=0}^\infty \frac{q^{i^2+i}(-q^6;q^6)_i}{(q^2;q^2)_{2i+1}} =2q\frac{J_{12}J_{24}^5J_{4,24}}{J_6J_{2,24}^2J_{6,24}^2J_{10,24}^2}. \quad \text{(by \eqref{MSZ124})} \label{36-S2}
\end{align}
Substituting \eqref{36-S1} and \eqref{36-S2} into \eqref{36-S}, we obtain \eqref{eq-exam6-1}. 

Now recall the identity \cite[Eq.\ (3.12)]{Xia-Yao}
\begin{align}\label{eq-Xia-Yao}
\frac{1}{J_1J_3}=\frac{J_8^2J_{12}^5}{J_2^2J_4J_6^4J_{24}^2}+q\frac{J_4^5J_{24}^2}{J_2^4J_6^2J_8^2J_{12}}.
\end{align}
Substituting it into \eqref{eq-exam6-2}, we obtain the 2-dissection formula \eqref{eq-exam6-1}.
\end{proof}

We can give a different proof if we sum over $i$ first.
\begin{proof}[Second proof of Theorem \ref{thm-exam-6}]
We have by \eqref{euler} that
\begin{align}\label{exam6-2nd-proof}
    &\sum_{i,j\geq 0} \frac{q^{i^2-3ij+3j^2}}{(q^2;q^2)_i(q^6;q^6)_j}=\sum_{j=0}^\infty \frac{q^{3j^2}}{(q^6;q^6)_j}\sum_{i=0}^\infty \frac{q^{i^2-i}\cdot q^{(1-3j)i}}{(q^2;q^2)_i} \nonumber \\
    &=\sum_{i=0}^\infty \frac{q^{3j^2}}{(q^6;q^6)_j}(-q^{1-3j};q^2)_\infty =T_0(q)+T_1(q),
\end{align}
where $T_0(q)$ and $T_1(q)$ correspond to the sums with $j$ being even and odd, respectively. We have
\begin{align}
&T_0(q)=\sum_{j=0}^\infty \frac{q^{12j^2}}{(q^6;q^6)_{2j}}(-q^{1-6j};q^2)_\infty \label{eq-T0} \\
&=(-q;q^2)_\infty \sum_{j=0}^\infty \frac{q^{3j^2}(-q;q^2)_{3j}}{(q^6;q^6)_{2j}}=\frac{J_2^2J_{24}^3}{J_1J_4J_{3,24}J_{4,24}J_{9,24}}, \quad \text{(by \eqref{lem-W-1})}  \nonumber \\
&T_1(q)=\sum_{j=0}^\infty \frac{q^{3(2j+1)^2}}{(q^6;q^6)_{2j+1}}(-q^{-2-6j};q^2)_\infty \label{eq-T1} \\
&=2q(-q^2;q^2)_\infty \sum_{j=0}^\infty \frac{q^{3j^2+3j}(-q^2;q^2)_{3j+1}}{(q^6;q^6)_{2j+1}} =2q\frac{J_4J_{24}^3J_{4,24}}{J_2J_{2,24}J_{6,24}^2J_{10,24}}. \quad \text{(by \eqref{lem-W-2})} \nonumber
\end{align}
Substituting \eqref{eq-T0} and \eqref{eq-T1} into \eqref{exam6-2nd-proof}, we obtain the first expression in Theorem \ref{thm-exam-6} and hence the whole theorem.
\end{proof}

\subsection{Examples 7 and 8}

Example 7 corresponds to 
\begin{align}
    A=
    \begin{pmatrix}
        2 & 1\\
        3 & 2
    \end{pmatrix} , \quad 
    b\in \bigg\{ 
    \begin{pmatrix}
        0 \\ 0
    \end{pmatrix} ,
    \begin{pmatrix}
        1 \\ 3
    \end{pmatrix} ,
    \begin{pmatrix}
        2 \\ 3
    \end{pmatrix} 
    \bigg\} , \quad
    d=(1,3) .\label{exam-4}
\end{align}
As pointed out by Mizuno \cite{Mizuno}, the modularity follows from three conjectural identities of Kanade and Russell \cite{K2015} (in the form by Kur\c{s}ung\"{o}z \cite{Kursungoz-AC}):
\begin{align}
 &\sum_{i,j\ge 0}\frac{q^{i^2+3ij+3j^2}}{(q;q)_i(q^3;q^3)_j}
        =\frac{1}{(q,q^3,q^6,q^8;q^9)_\infty}, \label{KR-1} \\
&\sum_{i,j\ge 0}\frac{q^{i^2+3ij+3j^2+i+3j}}{(q;q)_i(q^3;q^3)_j}
        =\frac{1}{(q^2,q^3,q^6,q^7;q^9)_\infty}, \label{KR-2} \\
&\sum_{i,j\ge 0}\frac{q^{i^2+3ij+3j^2+2i+3j}}{(q;q)_i(q^3;q^3)_j}
        =\frac{1}{(q^3,q^4,q^5,q^6;q^9)_\infty}. \label{KR-3} 
\end{align}
Though we were not able to prove them, we find the following new representations for the product sides, which can be shown equivalent to the above expressions using the method in the work of Frye and Garvan \cite{Frye-Garvan}.
\begin{conj}\label{WW-conj-1}
We have
    \begin{align}
        &\sum_{i,j\ge 0}\frac{q^{i^2+3ij+3j^2}}{(q;q)_i(q^3;q^3)_j}
        =
        \frac{(q^3,q^6;q^9)_\infty^2}{(q^4,q^5;q^9)_\infty^2(q,q^2,q^7,q^8;q^9)_\infty}-q^2\frac{(q,q^8;q^9)_\infty^2}{(q^4,q^5;q^9)_\infty^3(q^3,q^6;q^9)_\infty},
        \\
        &\sum_{i,j\ge 0}\frac{q^{i^2+3ij+3j^2+i+3j}}{(q;q)_i(q^3;q^3)_j}
        =
        \frac{(q^3,q^6;q^9)_\infty^2}{(q^2,q^4,q^5,q^7;q^9)_\infty^2}
        -q^2\frac{(q,q^8;q^9)_\infty^3}{(q^4,q^5;q^9)_\infty^3(q^2,q^3,q^6,q^7;q^9)_\infty},
        \\
         &\sum_{i,j\ge 0}\frac{q^{i^2+3ij+3j^2+2i+3j}}{(q;q)_i(q^3;q^3)_j}
         =
         \frac{(q^2,q^7;q^9)_\infty^2}{(q^4,q^5;q^9)_\infty^2(q,q^3,q^6,q^8;q^9)_\infty}
         -q\frac{(q,q^2,q^7,q^8;q^9)_\infty}{(q^4,q^5;q^9)_\infty^3(q^3,q^6;q^9)_\infty}.
    \end{align}
\end{conj}

Example 8 (the dual of Example 7) corresponds to
\begin{align}
    A= \begin{pmatrix}
        2 & -1\\
        -3 & 2
    \end{pmatrix}, \quad
    b\in \bigg\{
    \begin{pmatrix}
        0 \\ 0
    \end{pmatrix},
    \begin{pmatrix}
        -1 \\ 3
    \end{pmatrix},
    \begin{pmatrix}
        1 \\ 0
    \end{pmatrix}
    \bigg\}, \quad
    d=(1,3).
\end{align}
We propose the following conjecture, which justify the modularity of this example.
\begin{conj}\label{WW-conj-2}
   We have
   \begin{align}
       &\sum_{i,j\ge 0}\frac{q^{i^2-3ij+3j^2}}{(q;q)_i(q^3;q^3)_j}
        =
        \frac{1}{(q,q^3,q^6,q^8;q^9)_\infty^2}+q\frac{1}{(q^3,q^6;q^9)_\infty^2(q^2,q^4,q^5,q^7;q^9)_\infty} ,
        \\
        &\sum_{i,j\ge 0}\frac{q^{i^2-3ij+3j^2-i+3j}}{(q;q)_i(q^3;q^3)_j}
        =
        \frac{1}{(q^2,q^3,q^6,q^7;q^9)_\infty^2}+\frac{1}{(q^3,q^6;q^9)_\infty^2(q,q^4,q^5,q^8;q^9)_\infty} ,
        \\
        &\sum_{i,j\ge 0}\frac{q^{i^2-3ij+3j^2+i}}{(q;q)_i(q^3;q^3)_j}
        =
        \frac{(q^2,q^7;q^9)_\infty}{(q,q^3,q^6,q^8;q^9)_\infty^2(q^4,q^5;q^9)}
        -2q\frac{1}{(q^3,q^4,q^5,q^6;q^9)_\infty^2} .
   \end{align}
\end{conj}
There seems to be no single product representations for the Nahm sums involved in the above identities.

It is worth to mention that Kanade and Russell \cite{K2015} also conjectured the following companion result (in the form by Kur\c{s}ung\"{o}z \cite{Kursungoz-AC}):
\begin{align}\label{KR-4}
     &\sum_{i,j\ge 0}\frac{q^{i^2+3ij+3j^2+i+2j}}{(q;q)_i(q^3;q^3)_j}
        =
        \frac{1}{(q^2,q^3,q^5,q^8;q^9)_\infty}.
\end{align}
Note that the right side is not modular. Li and Wang \cite[Conjecture 6.4]{LW} conjectured the following identity:
\begin{align}\label{eq-LW-conj}
\sum_{i,j\geq 0} \frac{q^{i^2-3ij+3j^2+j}}{(q;q)_i(q^3;q^3)_j}=\frac{(q^6;q^9)_\infty}{(q,q^2,q^2,q^4,q^5,q^5;q^6)_\infty}.
\end{align}
Interestingly, this happens to be the dual case of \eqref{KR-4}. 

Alert readers may wonder whether we can find alternative expressions for \eqref{KR-4} and \eqref{eq-LW-conj} similar to those in Conjectures \ref{WW-conj-1} and \ref{WW-conj-2}. The answer is yes. We propose the following conjecture, which provides equivalent product representations for \eqref{KR-4} and \eqref{eq-LW-conj}. 
\begin{conj}\label{WW-conj-3}
We have
\begin{align}
  &\sum_{i,j\ge 0}\frac{q^{i^2+3ij+3j^2+i+2j}}{(q;q)_i(q^3;q^3)_j}=
 \frac{(q^2;q^9)_\infty(q^7;q^9)_\infty^2}{(q,q^3,q^4;q^9)_\infty(q^5,q^8;q^9)_\infty^2}
        -q\frac{(q,q^7;q^9)_\infty}{(q^3;q^9)_\infty(q^4;q^9)_\infty^2(q^5;q^9)_\infty^3}, \label{KR-4-new} \\
&\sum_{i,j\geq 0} \frac{q^{i^2-3ij+3j^2+j}}{(q;q)_i(q^3;q^3)_j} 
=\frac{(q^6,q^7;q^9)_\infty}{(q^5,q^8;q^9)_\infty^3(q,q^4;q^9)^2}
-q\frac{(q^6;q^9)_\infty}{(q^2,q^8;q^9)_\infty(q^4;q^9)_\infty^3(q^5;q^9)_\infty^4}. \label{eq-LW-conj-new}
\end{align}
\end{conj}
In contrast to Conjecture \ref{WW-conj-1}, it is not straightforward to use the method in \cite{Frye-Garvan} to show the equivalence of this conjecture and the formulas in \eqref{KR-4} and \eqref{eq-LW-conj}.

\subsection{Examples 9 and 10}
Example 9 corresponds to
\begin{align}
    A=\begin{pmatrix}
        3 & 2\\
        4 & 4
    \end{pmatrix} , \quad
    b \in \bigg\{ 
    \begin{pmatrix}
        -1/2 \\ 0
    \end{pmatrix} ,
    \begin{pmatrix}
        1/2 \\ 2
    \end{pmatrix}
    \bigg\} , \quad
    d = (1,2) . \label{exam-5}
\end{align}
Mizuno proved its modularity by establishing two Rogers-Ramanujan type identities (see \cite[(47) and (48)]{Mizuno}).

Example 10 (the dual of Example 9) corresponds to
\begin{align}
    A= \begin{pmatrix}
    1 & -1/2 \\
    -1 & 3/4 
    \end{pmatrix} ,\quad
    b \in \bigg\{
    \begin{pmatrix}
        -1/2 \\ 1/2
    \end{pmatrix} , 
    \begin{pmatrix}
        -1/2 \\ 1
    \end{pmatrix} \bigg\} , \quad
    d=(1,2) .
\end{align}
To prove its modularity, we establish the following theorem.
\begin{theorem}
    We have
    \begin{align}
        \sum_{i,j\ge 0}\frac{q^{2i^2-4ij+3j^2-2i+2j}}{(q^4;q^4)_i(q^8;q^8)_j}
        &=\frac{(-1;q^4)_\infty(q^2,q^3,q^5;q^5)_\infty}{(q,q^3,q^4;q^4)_\infty},
        \label{A1(2,-4,3).1}
        \\
        \sum_{i,j\ge 0}\frac{q^{2i^2-4ij+3j^2-2i+4j}}{(q^4;q^4)_i(q^8;q^8)_j}
        &=
        \frac{(-1;q^4)_\infty(q,q^4,q^5;q^5)_\infty}{(q,q^3,q^4;q^4)_\infty}.
        \label{A1(2,-4,3).2}
    \end{align}
\end{theorem}

\begin{proof}
  By \eqref{euler} we have
    \begin{align}
       & F(u,v)=F(u,v;q)
        :=\sum_{i,j\ge 0}\frac{q^{2i^2-4ij+3j^2}u^iv^j}{(q^4;q^4)_i(q^8;q^8)_j}
        =
        \sum_{j\ge 0}\frac{q^{3j^2}v^j}{(q^8;q^8)_j}
        \sum_{i\ge 0}\frac{q^{2i^2-2i}(q^{-4j+2}u)^i}{(q^4;q^4)_i} \nonumber
        \\
        &=\sum_{j\ge 0}\frac{q^{3j^2}v^j(-q^{-4j+2}u;q^4)_\infty}{(q^8;q^8)_j}
        =
        (-q^{2}u;q^4)_\infty\sum_{j\ge 0}\frac{q^{3j^2}v^j(-q^{-4j+2}u;q^4)_j}{(q^8;q^8)_j}\label{F1}.
    \end{align}
    
    Setting $(u,v)=(q^{-2},q^2)$, we have by $\eqref{F1}$ and $\eqref{S20}$ that
    \begin{align*}
       & F(q^{-2},q^2)
        =(-1;q^4)_\infty\sum_{j\ge 0}\frac{q^{3j^2+2j}(-q^{-4j};q^4)_j}{(q^8;q^8)_j}
        =
        (-1;q^4)_\infty\sum_{j\ge 0}\frac{q^{j^2}(-q^4;q^4)_j}{(q^8;q^8)_j}\notag
        \\
        &=(-1;q^4)_\infty
        \sum_{j\ge 0}\frac{q^{j^2}}{(q^4;q^4)_j}
        =
        \frac{(-1;q^4)_\infty(q^2,q^3,q^5;q^5)_\infty}{(q,q^3,q^4;q^4)_\infty}.
    \end{align*}
    This proves $\eqref{A1(2,-4,3).1}$. 

     Setting $(u,v)=(q^{-2},q^4)$, we have by $\eqref{F1}$ and $\eqref{S16}$ that
    \begin{align*}
       & F(q^{-2},q^4)
        =(-1;q^4)_\infty\sum_{j\ge 0}\frac{q^{3j^2+4j}(-q^{-4j};q^4)_j}{(q^8;q^8)_j}
        =
        (-1;q^4)_\infty\sum_{j\ge 0}\frac{q^{j^2+2j}(-q^4;q^4)_j}{(q^8;q^8)_j}\notag
        \\
        &=(-1;q^4)_\infty
        \sum_{j\ge 0}\frac{q^{j^2+2j}}{(q^4;q^4)_j}
        =
        \frac{(-1;q^4)_\infty(q,q^4,q^5;q^5)_\infty}{(q,q^3,q^4;q^4)_\infty}.
    \end{align*}
    This proves $\eqref{A1(2,-4,3).2}$. 
\end{proof}

\subsection{Examples 11 and 12}
Example 11 corresponds to 
\begin{align}
    A=
    \begin{pmatrix}
        3/2 & 1/2\\
        1 &  1
    \end{pmatrix},\quad
    b\in \bigg\{ 
    \begin{pmatrix}
        -1\\1
    \end{pmatrix} , 
    \begin{pmatrix}
        -1/2 \\ 0
    \end{pmatrix} ,
    \begin{pmatrix}
        0 \\ 1
    \end{pmatrix}
    \bigg\},\quad
    d=(1,2) .\label{exam-6}
\end{align}
Mizuno \cite[(50)--(53)]{Mizuno} established some identities to prove its modularity.

Example 12 (the dual of Example 11) corresponds to
\begin{align}
    A= \begin{pmatrix}
    1 & -1/2 \\
    -1 & 3/2
    \end{pmatrix} ,\quad
    b \in \bigg\{ 
    \begin{pmatrix}
        -3/2 \\ 5/2
    \end{pmatrix} ,
    \begin{pmatrix}
        -1/2 \\ 1/2
    \end{pmatrix} , 
    \begin{pmatrix}
        -1/2 \\3/2
    \end{pmatrix}
    \bigg\} , \quad
    d=(1,2) . 
\end{align}

To prove its modularity, we establish the following theorem.

\begin{theorem}
    We have 
    \begin{align}
        \sum_{i,j\ge 0}\frac{q^{\frac{1}{2}i^2-ij+\frac{3}{2}j^2-\frac{3}{2}i+\frac{5}{2}j}}{(q;q)_i(q^2;q^2)_j}
        &=
        \frac{2q^{-1}(-q;q)_\infty}{(q,q^4;q^5)_\infty}, 
        \label{A1(1/2,-1,3/2).1}
        \\
        \sum_{i,j\ge 0}\frac{q^{\frac{1}{2}i^2-ij+\frac{3}{2}j^2-\frac{1}{2}i+\frac{1}{2}j}}{(q;q)_i(q^2;q^2)_j}
        &=\frac{2(-q;q)_\infty}{(q,q^4;q^5)_\infty},
        \label{A1(1/2,-1,3/2).2}
        \\
        \sum_{i,j\ge 0}\frac{q^{\frac{1}{2}i^2-ij+\frac{3}{2}j^2-\frac{1}{2}i+\frac{3}{2}j}}{(q;q)_i(q^2;q^2)_j}
        &=\frac{2(-q;q)_\infty}{(q^2,q^3;q^5)_\infty}.
        \label{A1(1/2,-1,3/2).3}
    \end{align}
\end{theorem}
    \begin{proof}
By \eqref{euler} we have
        \begin{align}
           &F(u,v)= F(u,v;q)
            :=\sum_{i,j\ge 0}\frac{q^{\frac{1}{2}i^2-ij+\frac{3}{2}j^2}u^iv^j}{(q;q)_i(q^2;q^2)_j}
            =
            \sum_{j\ge 0}\frac{q^{\frac{3}{2}j^2}v^j}{(q^2;q^2)_j}
            \sum_{i\ge 0}\frac{q^{\frac{1}{2}(i^2-i)}(q^{-j+\frac{1}{2}}u)^i}{(q;q)_i}\notag
            \\
            &=
            \sum_{j\ge 0}\frac{q^{\frac{3}{2}j^2}v^j(-q^{-j+\frac{1}{2}}u;q)_\infty}{(q^2;q^2)_j}
            =
            (-q^{\frac{1}{2}}u;q)_\infty\sum_{j\ge 0}\frac{q^{\frac{3}{2}j^2}v^j(-q^{-j+\frac{1}{2}}u;q)_j}{(q^2;q^2)_j}.
            \label{F2}
        \end{align}

Setting $(u,v)=(q^{-\frac{3}{2}},q^{\frac{5}{2}})$, we have by $\eqref{F2}$ and $\eqref{RR-variant}$ that
\begin{align*}
&F(q^{-\frac{3}{2}},q^{\frac{5}{2}})
=(-q^{-1};q)_\infty
 \sum_{j\ge 0}\frac{q^{\frac{3}{2}j^2+\frac{5}{2}j}(-q^{-j-1};q)_j}{(q^2;q^2)_j}\notag
            \\
&=\frac{(-q^{-1};q)_\infty}{1+q}\sum_{j\ge 0}\frac{q^{j^2+j}(-q;q)_{j+1}}{(q^2;q^2)_j}
=\frac{2q^{-1}(-q;q)_\infty}{(q,q^4;q^5)_\infty}.
\end{align*}
This proves $\eqref{A1(1/2,-1,3/2).1}$. 

Setting $(u,v)=(q^{-\frac{1}{2}},q^{\frac{1}{2}})$, we have by $\eqref{F2}$ and $\eqref{RR}$ that
\begin{align*}
& F(q^{-\frac{1}{2}},q^{\frac{1}{2}})
 =(-1;q)_\infty
 \sum_{j\ge 0}\frac{q^{\frac{3}{2}j^2+\frac{1}{2}j}(-q^{-j};q)_j}{(q^2;q^2)_j}\notag
            \\
            &=(-1;q)_\infty\sum_{j\ge 0}\frac{q^{j^2}(-q;q)_j}{(q^2;q^2)_j}
            =2
            (-q;q)_\infty\sum_{j\ge 0}\frac{q^{j^2}}{(q;q)_j}
            =\frac{2(-q;q)_\infty}{(q,q^4;q^5)_\infty}. 
\end{align*}
This proves $\eqref{A1(1/2,-1,3/2).2}$.

Setting $(u,v)=(q^{-\frac{1}{2}},q^{\frac{3}{2}})$, we have by $\eqref{F2}$ and $\eqref{RR}$ that
\begin{align*}
     &       F(q^{-\frac{1}{2}},q^{\frac{3}{2}})
            =(-1;q)_\infty
            \sum_{j\ge 0}\frac{q^{\frac{3}{2}j^2+\frac{3}{2}j}(-q^{-j};q)_j}{(q^2;q^2)_j}\notag
            \\
            &=(-1;q)_\infty\sum_{j\ge 0}\frac{q^{j^2+j}(-q;q)_j}{(q^2;q^2)_j}
            =2
            (-q;q)_\infty\sum_{j\ge 0}\frac{q^{j^2+j}}{(q;q)_j}
            =\frac{2(-q;q)_\infty}{(q^2,q^3;q^5)_\infty}. 
\end{align*}
This proves $\eqref{A1(1/2,-1,3/2).3}$. 
\end{proof}

\subsection{Examples 13 and 14}
Example 13 corresponds to 
    \begin{align}
        A=\begin{pmatrix}
            3 & 1\\
            4 & 2
        \end{pmatrix} , \quad
        b \in \bigg\{
        \begin{pmatrix}
            -1/2 \\ 0
        \end{pmatrix} , 
        \begin{pmatrix}
            3/2 \\ 4
        \end{pmatrix}
        \bigg\} , \quad
        d=(1,4). \label{exam-7}
    \end{align}
As pointed out by Mizuno, the modularity follows from two identities in the work of Kur\c{s}ung\"oz \cite[(21) and (22)]{Kursungoz-JCTA}.

Example 14 (the dual of Example 13) corresponds to    
    \begin{align}
        A=\begin{pmatrix}
            1 & -1/2\\
            -2 & 3/2
        \end{pmatrix} ,  \quad
        b \in \bigg\{
        \begin{pmatrix}
            -1/2 \\ 1
        \end{pmatrix} ,
        \begin{pmatrix}
            -1/2 \\ 3
        \end{pmatrix}
        \bigg\}, \quad 
        d=(1,4).
    \end{align}
The modularity follows from the identities \cite[Eqs.\ (3.79) and (3.81)]{LW}:
        \begin{align}
            \sum_{i,j\ge 0}\frac{q^{\frac{1}{2}i^2-2ij+3j^2-\frac{1}{2}i+j}}{(q;q)_i(q^4;q^4)_j}
            &=\frac{2}{(q;q^2)_\infty(q,q^4,q^7;q^8)_\infty} ,\label{t1-14-1} 
            \\
            \sum_{i,j\ge 0}\frac{q^{\frac{1}{2}i^2-2ij+3j^2-\frac{1}{2}i+3j}}{(q;q)_i(q^4;q^4)_j}
            &=\frac{2}{(q;q^2)_\infty(q^3,q^4,q^5;q^8)_\infty} .\label{t1-14-2}
        \end{align}

We also find some new non-modular identities for the following choices of $(A,b,d)$ where $A$ and $d$ are the same with Example 13 :
\begin{align}
    A= \begin{pmatrix}
        3 & 1\\
        4 & 2
    \end{pmatrix}, \quad
    b \in \bigg\{ \begin{pmatrix}
        1/2 \\ 2
    \end{pmatrix}, 
    \begin{pmatrix}
        -1/2 \\ 2
    \end{pmatrix} ,
    \begin{pmatrix}
        -5/2 \\ 0
    \end{pmatrix}
    \bigg\} , \quad
    d=(1,4) . \label{new-exam-13}
\end{align}
For the first and second choices of $b$ the corresponding identities are \cite[Eq.\ (3.78)]{LW} and \cite[Eq.\ (28)]{Kursungoz-JCTA}:
\begin{align}
       &\sum_{i,j\ge 0}\frac{q^{\frac{3}{2}i^2+4ij+4j^2+\frac{1}{2}i+2j}}{(q;q)_i(q^4;q^4)_j}
        =
        \frac{1}{(q^2,q^3,q^7;q^8)_\infty}, \\
     &\sum_{i,j\ge 0}\frac{q^{\frac{3}{2}i^2+4ij+4j^2-\frac{1}{2}i+2j}}{(q;q)_i(q^4;q^4)_j}
        =
        \frac{1}{(q,q^5,q^6;q^8)_\infty}.
    \end{align}
For the third choice of $b$ we find the following new identity.
\begin{theorem}\label{thm-new-exam13}
    We have 
    \begin{align}
        &\sum_{i,j\ge 0}\frac{q^{\frac{3}{2}i^2+4ij+4j^2-\frac{5}{2}i}}{(q;q)_i(q^4;q^4)_j}
        =q^{-1}(1+q)\frac{1}{(q,q^4,q^7;q^8)_\infty}. \label{ex7-1-1}
    \end{align}
\end{theorem}
\begin{proof}
We have
    \begin{align}
        &\sum_{i,j\ge 0}\frac{q^{\frac{1}{2}i^2-\frac{5}{2}i}x^{i+2j}}{(q;q)_i(q^4;q^4)_j}
        =
        \frac{(-xq^{-2};q)_\infty}{(x^2;q^4)_\infty}
        =\frac{(-xq^{-2};q^2)_\infty(-xq^{-1};q^2)_\infty}{(x;q^2)_\infty(-x;q^2)_\infty}\notag
        \\
        &=(1+xq^{-2})\frac{(-xq^{-1};q^2)_\infty}{(x;q^2)_\infty}
        =(1+xq^{-2})\sum_{m\ge 0}\frac{(-q^{-1};q^2)_mx^m}{(q^2;q^2)_m} \notag
        \\
        &=1+\sum_{m\ge 1}\left(\frac{(-q^{-1};q^2)_m}{(q^2;q^2)_m}+q^{-2}\frac{(-q^{-1};q^2)_{m-1}}{(q^2;q^2)_{m-1}}\right)x^m .
    \end{align}
    Comparing the coefficients of $x^m$ on both sides, we deduce that 
    \begin{align}
        &\sum_{i+2j=m}\frac{q^{\frac{1}{2}i^2-\frac{5}{2}i}}{(q;q)_i(q^4;q^4)_j}
        =
        \frac{(-q^{-1};q^2)_m}{(q^2;q^2)_m}+q^{-2}\frac{(-q^{-1};q^2)_{m-1}}{(q^2;q^2)_{m-1}} . 
    \end{align}
    We have
    \begin{align}
        &F(u;q):=\sum_{i,j\ge 0}\frac{q^{\frac{3}{2}i^2+4ij+4j^2-\frac{5}{2}i}u^{i+2j}}{(q;q)_i(q^4;q^4)_j}
        =
        \sum_{m\ge 0}q^{m^2}u^m
        \sum_{i+2j=m}\frac{q^{\frac{1}{2}i^2-\frac{5}{2}i}}{(q;q)_i(q^4;q^4)_j}
   \nonumber     \\
        &=1+\sum_{m\ge 1}q^{m^2}u^m\left(\frac{(-q^{-1};q^2)_m}{(q^2;q^2)_m}+q^{-2}\frac{(-q^{-1};q^2)_{m-1}}{(q^2;q^2)_{m-1}}\right)
    \nonumber    \\
        &=\sum_{m\ge 0} \frac{q^{m^2}u^m(-q^{-1};q^2)_m}{(q^2;q^2)_m}+\sum_{m \ge 0} \frac{q^{(m+1)^2-2}u^{m+1}(-q^{-1};q^2)_m}{(q^2;q^2)_m}
    \nonumber    \\
        &=\sum_{m\ge 0}\frac{(1+q^{2m-1}u)q^{m^2}u^m(-q^{-1};q^2)_m}{(q^2;q^2)_m} . \label{ex13-F1}
    \end{align}
    Setting $u=1$, we have by \eqref{ex13-F1} and \eqref{S36} that
    \begin{align*}
        &F(1;q)=\sum_{m\ge 0}\frac{(1+q^{2m-1})q^{m^2}(-q^{-1};q^2)_m}{(q^2;q^2)_m}\notag
        \\
        &=(1+q^{-1})\sum_{m\ge 0}\frac{q^{m^2}(-q;q^2)_m}{(q^2;q^2)_m}=q^{-1}(1+q)\frac{1}{(q,q^4,q^7;q^8)_\infty}. \qedhere
    \end{align*}
\end{proof}

The following cases are dual to \eqref{new-exam-13},  and they share the same $A,d$ with Example 14: 
\begin{align}
    A= \begin{pmatrix}
        1 & -1/2\\
        -2 & 3/2
    \end{pmatrix}, \quad
    b \in \bigg\{
    \begin{pmatrix}
        -1/2 \\ 2
    \end{pmatrix},
    \begin{pmatrix}
        -3/2 \\ 4
    \end{pmatrix} ,
    \begin{pmatrix}
        -5/2 \\ 5
    \end{pmatrix} 
    \bigg\} , \quad
    d=(1,4) .
\end{align}
\begin{theorem}\label{thm-new-exam14}
    We have 
    \begin{align}
        &\sum_{i,j \ge 0} \frac{q^{\frac{1}{2}i^2-2ij+3j^2-\frac{1}{2}i+2j}}{(q;q)_i(q^4;q^4)_j}
        =
        2\frac{(-q;q)_\infty}{(q^2,q^3,q^7;q^8)_\infty}, \label{14-2-1}
        \\
        &\sum_{i,j\ge 0}\frac{q^{\frac{1}{2}i^2-2ij+3j^2-\frac{3}{2}i+4j}}{(q;q)_i(q^4;q^4)_j}
        =
        2q^{-1}\frac{(-q;q)_\infty}{(q,q^5,q^6;q^8)_\infty}, \label{14-2-2}
        \\
        &\sum_{i,j\ge 0}\frac{q^{\frac{1}{2}i^2-2ij+3j^2-\frac{5}{2}i+5j}}{(q;q)_i(q^4;q^4)_j}
        =
        2q^{-3}(1+q)\frac{(-q;q)_\infty}{(q,q^4,q^7;q^8)_\infty}. \label{14-2-3}
    \end{align}
\end{theorem}

\begin{proof}
    We have
    \begin{align}
        &F(u,v)=F(u,v;q):=\sum_{i,j\ge 0}\frac{q^{\frac{1}{2}i^2-2ij+3j^2}u^iv^j}{(q;q)_i(q^4;q^4)_j}
        =\sum_{j\ge 0}\frac{q^{3j^2}v^j}{(q^4;q^4)_j}\sum_{i\ge 0}\frac{q^{\frac{1}{2}(i^2-i)}(uq^{\frac{1}{2}-2j})^i}{(q;q)_i}
    \nonumber    \\
        &=\sum_{j\ge 0}\frac{q^{3j^2}v^j(-uq^{\frac{1}{2}-2j};q)_\infty}{(q^4;q^4)_j}
        =(-uq^{\frac{1}{2}};q)_\infty\sum_{j\ge 0}\frac{q^{3j^2}v^j(-uq^{\frac{1}{2}-2j};q)_{2j}}{(q^4;q^4)_j} \nonumber \\
        &=(-uq^{\frac{1}{2}};q)_\infty\sum_{j\ge 0}\frac{q^{j^2}u^{2j}v^j(-u^{-1}q^\frac{1}{2};q)_{2j}}{(q^4;q^4)_j} .
        \label{F14-2}
    \end{align}    
    Setting $(u,v) =(q^{-\frac{1}{2}},q^2)$, we have by \eqref{F14-2} and \eqref{Go2.24} that
    \begin{align*}
        &F(q^{-\frac{1}{2}},q^{2})=(-1;q)_\infty\sum_{j\ge 0}\frac{q^{j^2+j}(-q;q)_{2j}}{(q^4;q^4)_j}
        =2(-q;q)_\infty\sum_{j\ge 0}\frac{q^{j(j+1)}(-q;q^2)_j(-q^2;q^2)_j}{(-q^2;q^2)_j(q^2;q^2)_j}\\
        &=2(-q;q)_\infty\sum_{j\ge 0}\frac{q^{j(j+1)}(-q;q^2)_j}{(q^2;q^2)_j}
        =2\frac{(-q;q)_\infty}{(q^2,q^3,q^7;q^8)_\infty} .
    \end{align*}
    This proves \eqref{14-2-1} .

    Setting $(u,v)=(q^{-\frac{3}{2}},q^4)$, we have by \eqref{F14-2} and \eqref{Lebesgue} that
    \begin{align*}
        &F(q^{-\frac{3}{2}},q^4)=(-q^{-1};q)_\infty\sum_{j\ge 0}\frac{q^{j^2+j}(-q^2;q)_{2j}}{(q^4;q^4)_j}
        \\
        &=
        (-q^{-1};q)_\infty\sum_{j\ge 0}\frac{q^{j^2+j}(-q^2;q^2)_j(-q^3;q^2)_j}{(-q^2;q^2)_j(q^2;q^2)_j}
        =(-q^{-1};q)_\infty\sum_{j\ge 0}\frac{q^{j^2+j}(-q^3;q^2)_j}{(q^2;q^2)_j}\\
        &=(-q^{-1};q)_\infty(-q^5;q^4)_\infty(-q^2;q^2)_\infty
        =2q^{-1}\frac{(-q;q)_\infty}{(q,q^5,q^6;q^8)_\infty} .
    \end{align*}
    This proves \eqref{14-2-2} .

    Setting $(u,v)=(q^{-\frac{5}{2}},q^5)$, we have
    \begin{align*}
        &F(q^{-\frac{5}{2}},q^5)= (-q^{-2};q)_\infty\sum_{j\ge 0}\frac{q^{j^2}(-q^3;q)_{2j}}{(q^4;q^4)_j}
        =q^{-2}(-q^{-1};q)_\infty\sum_{j\ge 0}\frac{q^{j^2}(-q^2;q)_{2j+1}}{(q^4;q^4)_j} \nonumber \\
        &=q^{-2}(-q^{-1};q)_\infty\sum_{j\ge 0}\frac{(1+q^{2j+2})q^{j^2}(-q^3;q^2)_j}{(q^2;q^2)_j}
        \\
    &=q^{-2}(-q^{-1};q)_\infty\left(\sum_{j=0}^\infty \frac{q^{j^2}(-q^3;q^2)_j}{(q^2;q^2)_j}+q\sum_{j=0}^\infty \frac{q^{(j+1)^2}(-q^3;q^2)_j}{(q^2;q^2)_j} \right)\nonumber \\
    &=q^{-2}(-q^{-1};q)_\infty\left(\sum_{j=0}^\infty \frac{q^{j^2}(-q^3;q^2)_j}{(q^2;q^2)_j}+q\sum_{j=1}^\infty \frac{q^{j^2}(-q^3;q^2)_{j-1}}{(q^2;q^2)_{j-1}}\right) \nonumber \\
    &=q^{-2}(-q^{-1};q)_\infty\sum_{j=0}^\infty  \frac{q^{j^2}(-q^3;q^2)_j}{(q^2;q^2)_j} \left(1+\frac{q(1-q^{2j})}{1+q^{2j+1}}\right) \nonumber \\
    &=2q^{-3}(1+q)(-q;q)_\infty\sum_{j=0}^\infty \frac{q^{j^2}(-q;q^2)_j}{(q^2;q^2)_j}      =2q^{-3}(1+q)\frac{(-q;q)_\infty}{(q,q^4,q^7;q^8)_\infty}.  \quad \text{(by \eqref{S36})}
    \end{align*}
This proves \eqref{14-2-3}.
\end{proof}

\subsection*{Acknowledgements}
This work was supported by the National Natural Science Foundation of China (12171375).

\end{document}